\documentclass[preprint,12pt]{elsarticle}
\usepackage{amsthm,amsmath,amssymb}
\usepackage{graphicx}
\usepackage{float}
\usepackage{mathtools}
\usepackage{epstopdf}
\usepackage{epsfig}
\usepackage{subfigure}
\usepackage{blkarray}
\usepackage{xcolor}

\theoremstyle{plain}
\newtheorem{theorem}{Theorem}
\newtheorem{lemma}[theorem]{Lemma}
\newtheorem{corollary}[theorem]{Corollary}
\newtheorem{proposition}[theorem]{Proposition}

\theoremstyle{definition}
\newtheorem{definition}[theorem]{Definition}

\newtheorem{example}[theorem]{Example}

\newtheorem{problem}[theorem]{Problem}

\theoremstyle{remark}
\newtheorem{remark}[theorem]{Remark}

\title{Partial-twuality polynomials of matrices}

\author{Qingying Deng$^{1}$, \ \ Xian'an Jin$^{2, 3}$, \ \ Qi Yan$^{4}$\footnote{Corresponding author.}\\
\footnotesize~~\\
			\footnotesize $^{1}$ School of Mathematics and Computational Science, Xiangtan University, P. R. China\\
           \footnotesize $^{2}$ School of Mathematical Sciences, Xiamen University, P. R. China\\
            \small $^{3}$School of Mathematics and Statistics, Qinghai Minzu University, P. R. China\\
			\footnotesize $^{4}$ School of Mathematics and Statistics, Lanzhou University, P. R. China\\
			\footnotesize Email: qingying@xtu.edu.cn, xajin@xmu.edu.cn, yanq@lzu.edu.cn,}
\date{}
\journal{arXiv}
\begin{document}
\begin{abstract}
The study of partial-twuality polynomials originates from the classical operations of geometric duality and Petrie duality on cellularly embedded graphs. These involutions generate the symmetric group $S_3$, and applying them to subsets of edges yields the notions of partial-(geometric) duality, partial-Petriality, and more generally, partial-twuality. In this paper, we generalize this theory of partial-twuality polynomials within the framework of matrix algebra. The key observation that the Euler genus of a bouquet under a partial-twuality can be expressed as a rank function of its adjacency matrix motivates and leads to the definition of a partial-twuality polynomial for an arbitrary square matrix over any field, thereby providing a universal algebraic counterpart to the topological polynomials. We then investigate basic properties of these polynomials, including product formulas, recursion relations, degrees, interpolation behaviors, and invariance and duality theorems under the matrix operations of pivoting and inversion. We conclude by posing some problems for further research.
\end{abstract}
\maketitle

\section{Introduction}

Geometric duality and Petrie duality have been extensively studied for cellularly embedded graphs, both in the plane and on surfaces of higher genus. Denoting these operations by $\delta$ and $\tau$, respectively, it is known that they are operations of order two that do not commute. A fundamental result is that their composition $\delta\tau$ has order three; consequently, they generate a group isomorphic to the symmetric group \(S_3\), consisting of six operators: the identity and the five non-trivial operators $\delta$, $\tau$, $\delta\tau$, $\tau\delta$, and $\tau\delta\tau$. Wilson \cite{SEW} first established that these operations yield an action of \(S_3\) on regular maps, and Lins \cite{LS} later extended this result to maps in general. The five non-trivial operators were subsequently termed twualities by Abrams and Ellis-Monaghan \cite{ABEL}.

Three connections between the Jones polynomials of (virtual) links and the topological Tutte polynomials of ribbon graphs have been established, based on different constructions of ribbon graphs from (virtual) link diagrams \cite{CP2007, C2G}. Chumtov \cite{CG} posed the concept of partial-duality of ribbon graphs and unified the above three relations.  Ellis-Monaghan and Moffatt \cite{EM1} have generalized this partial-duality construction to any of the five twuality operators, termed partial-twualities. Moreover, they introduced the more general notion of a twisted dual, obtained by allowing each edge to be independently assigned any one of the six operators.

Similar to the extensively studied genus polynomial in topological graph theory that enumerates all embeddings of a given graph by genus, Gross, Mansour and Tucker \cite{GMT, GMT2} introduced the partial-twuality polynomial for the operators $\delta$, $\tau$, $\delta\tau$, $\tau\delta$, and $\tau\delta\tau$ that enumerates all partial-twualities of a ribbon graph by Euler genus. They studied various basic properties of these polynomials, including interpolation and log-concavity. Subsequently, these polynomials have been further investigated in several other papers \cite{Chenyi, Cheng, CG2, QYXJ}.

Delta-matroids generalize embedded graphs in the same way that matroids generalize graphs. The analogy between graph theory and matroid theory is paralleled by the relationship between embedded graphs and delta-matroids, and a significant part of research on delta-matroids consists in generalising results about embedded graphs. For instance, Chun, Moffatt, Noble, and Rueckriemen \cite{CISR} showed that the theory of delta-matroids naturally extends that of embedded graphs, in the sense that fundamental ribbon graph operations and concepts admit natural counterparts in the delta-matroid setting. In this way the concepts of partial-duality, partial-Petriality, and Euler genus can be interpreted in the language of delta-matroids as twist, loop complementation, and width, respectively \cite{CISR, CMNR}.

Chmutov and Vignes-Tourneret \cite{CG2} raised the question of whether the partial-twuality polynomial, and related conjectures, can be extended to general delta-matroids.  In \cite{QYXJ3, QYJ4}, Yan and Jin introduced and studied partial-twuality polynomials for delta-matroids; indeed, such polynomials can be defined for any set system.  They proved that for normal binary delta-matroids (a generalization for bouquets), the partial-twuality polynomials are determined by signed intersection graphs, which implies that these polynomials for normal binary delta-matroids can be defined directly on signed intersection graphs.
Recently, Cheng \cite{Cheng} showed that the Euler genus of an orientable bouquet is related to the rank of the adjacency matrix of its intersection graph and extended the concept of the partial-dual Euler genus polynomial from intersection graphs to all graphs.

The matrix operation of pivoting was introduced by Tucker \cite{Tucker} in 1960. Bouchet \cite{AB2} later proved that this operation corresponds to twisting in a delta-matroid, which is the delta-matroid analogue of the partial-dual of a ribbon graph. This establishes a pivotal connection between matrix theory, delta-matroids, and embedded graphs. Motivated by this deep connection, in this paper we introduce an analogue for matrices of the partial-twuality polynomials, which have been studied for embedded graphs and delta-matroids in recent years.

In Section~2 we recall the partial-twuality polynomial for ribbon graphs. We show how the Euler genus after a partial-twuality of a bouquet can be written as a rank function of the adjacency matrix over \(\mathrm{GF}(2)\) of the bouquet's intersection graft. This observation leads to the central definition of the paper: we lift these rank expressions to define the partial-twuality polynomial for an arbitrary square matrix over any field. In Section 3, we study the basic properties of the partial-twuality polynomials of matrices. We establish their multiplicative behavior and reduction formulas for isolated vertices, derive leaf-reduction recursions for grafts, and obtain degree bounds. In Section 4, we study the interpolation behavior of the partial-twuality polynomials of matrices. In Section 5, we establish that the partial-twuality polynomials possess important invariance and duality properties under matrix operations. Specifically, the partial-$\langle\delta\rangle$ polynomial is invariant under matrix pivoting, while matrix inversion interchanges the polynomials for $\langle\tau\delta\tau\rangle$ and $\langle\tau\rangle$, and those for $\langle\delta\tau\rangle$ and $\langle\tau\delta\rangle$.  In Section 6, we pose several open problems  for further research.

\section{Partial-twuality polynomials: from ribbon graphs to matrices}
\subsection{Ribbon graphs and partial-twuality}
In this paper we use the language of ribbon graphs, which are equivalent to cellularly embedded graphs (see \cite{EM}) and are defined as follows.

\begin{definition}[\cite{BR}]
A {\it ribbon graph} $G=(V(G), E(G))$ is a $($orientable or non-orientable$)$ surface with boundary,
represented as the union of two sets of topological discs: a set $V(G)$ of vertices and a set $E(G)$ of edges with the following properties.
\begin{description}
\item[\rm (1)] The vertices and edges intersect in disjoint line segments.
\item[\rm (2)] Each such line segment lies on the boundary of precisely one vertex and precisely one edge.
\item[\rm (3)] Every edge contains exactly two such line segments.
\end{description}
\end{definition}

For a ribbon graph $G$, we denote by $f(G)$, $v(G)$, $e(G)$, and $c(G)$ the number of its boundary components, vertices, edges, and connected components, respectively. We let \[\chi(G)=v(G)-e(G)+f(G),\] the usual \emph{Euler characteristic}, where $G$ is connected or not. The notation $\varepsilon(G)$ represents the \emph{Euler genus} of $G$, that is, \[\varepsilon(G)=2c(G)-\chi(G).\]
 Given an edge subset $A \subseteq E(G)$, the \emph{ribbon subgraph} $G \setminus A$ is obtained by deleting all edges in $A$ while retaining all vertices. The \emph{spanning ribbon subgraph} induced by $A$ is defined as $(V(G), A)=G\setminus A^c$, where $A^c \coloneqq E(G) \setminus A$.

 Chmutov's partial-duality \cite{CG} extends the classical notion of geometric duality to a subset of edges. This operation plays a role in knot theory, topological graph theory, graph polynomials, and quantum field theory \cite{EM1, MFF}. In delta-matroids, partial duality corresponds to twisting \cite{CMNR}. In this paper, we define the partial-dual directly on ribbon graphs; for alternative constructions and perspectives, see \cite{EM}.

\begin{definition}[\cite{CG}]\label{def01}
For a ribbon graph $G$ and $A\subseteq E(G)$,  the \emph{partial-dual} $G^{\delta(A)}$ of $G$ with respect to $A$ is a ribbon graph obtained from $G$ by gluing a disc to $G$ along each boundary component of the spanning ribbon subgraph $(V (G), A)$ $($such discs will be the vertex-discs of $G^{\delta(A)})$, removing the interiors of all the vertex-discs of $G$ and keeping the edge ribbons unchanged.
\end{definition}

The \emph{geometric dual} $G^*$ of a ribbon graph $G$, which exchanges vertices and faces, is equal to $G^{\delta{(E(G))}}$.  Next, we introduce the Petrie dual of $G$, denoted by $G^{\times}$, which is also known as the Petrial. This concept originates from Wilson's work \cite{SEW}.

\begin{definition}[\cite{EM1}]
Let $G$ be a ribbon graph and $A\subseteq E(G)$. Then the \emph{partial-Petrial}, $G^{\tau(A)}$, of a ribbon graph $G$ with respect to $A$ is the ribbon graph obtained from $G$ by adding a half-twist to each edge in $A$.
\end{definition}

We first note a basic commutativity property: for distinct edges \(e\) and \(f\) of a ribbon graph, applying the partial Petrial \(\tau\) to \(e\) and the partial dual \(\delta\) to \(f\) are commuting operations. This commutativity fails, however, when the two operations are applied to the same edge. To formalize compositions of the operators \(\delta\) and \(\tau\) on a subset of edges, we introduce the following notation.

Let $G$ be a ribbon graph and $A \subseteq E(G)$. For any operator word $w = w_1w_2\cdots w_n$ over the alphabet $\{\delta, \tau\}$, we recursively define the ribbon graph
\begin{equation*}
G^{w(A)} \coloneqq (\cdots(G^{w_n(A)})^{w_{n-1}(A)} \cdots)^{w_1(A)}.
\end{equation*}

\begin{figure}
    \centering
\includegraphics[width=0.9\linewidth]{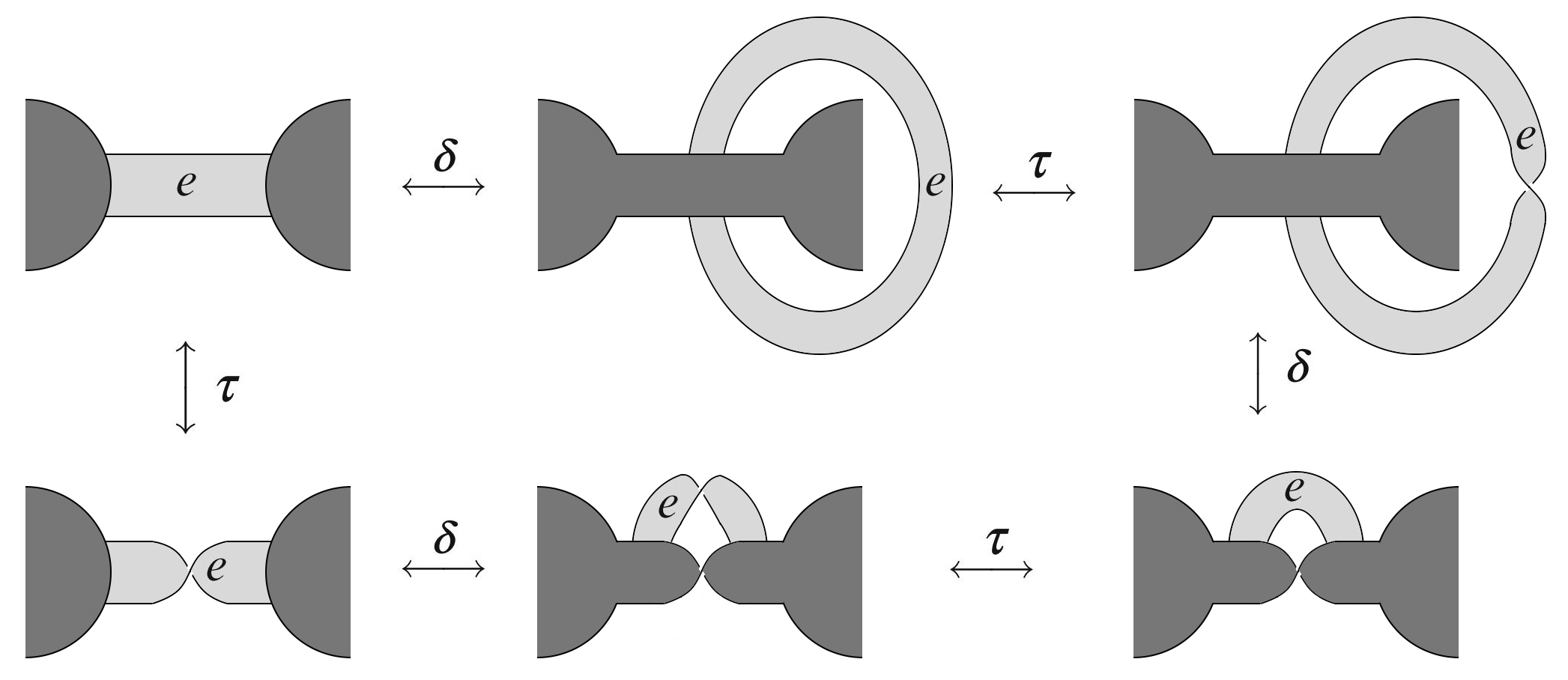}
   \caption{The actions of $\delta$ and $\tau$ on an edge $e$}
    \label{f_1}
\end{figure}

Figure \ref{f_1} illustrates the effect of applying the operations $\delta$ and $\tau$ to a single edge of a ribbon graph. Ellis-Monaghan and Moffatt \cite{EM1} proved that the partial-dual $\delta$ and partial-Petrial $\tau$ generate a symmetric group action $S_3$ on ribbon graphs, with the group presentation
\[
S_3 \cong \mathfrak{G}\coloneqq \langle \delta, \tau \mid \delta^2, \tau^2, (\delta\tau)^3 \rangle.
\]
Suppose $G$ is a ribbon graph, $A, B\subseteq E(G)$, and $\xi, \pi \in \mathfrak{G}$. Then we define $$G^{\xi(A)\pi(B)}\coloneqq(G^{\xi(A)})^{\pi(B)}.$$

\begin{proposition} [\cite{EM1}]\label{pro10}
  If, for all \( i \), we have \( \zeta_i \in \mathfrak{G} \) and \( B_i \subseteq E(G) \), then any expression of the form \( G^{\Pi \zeta_i(B_i)} \) is equal to \( G^{\Pi_{i=1}^6 \xi_i(A_i)} \), where the \( A_i \subseteq E(G) \) are pairwise disjoint with \( \bigcup_i A_i = E(G) \) and where \( \xi_1 = 1, \xi_2 = \tau, \xi_3 = \delta, \xi_4 = \tau\delta, \xi_5 = \delta\tau, \) and \( \xi_6 = \tau\delta\tau \in \mathfrak{G} \). Here, the terms in the product \( \prod \zeta_i(B_i) \) do not necessarily commute, while the terms in the product \( \prod_{i=1}^6 \xi_i(A_i) \) do commute with each other.
\end{proposition}

\subsection{The partial-twuality polynomial for bouquets}

We first recall the definition of the partial-twuality polynomial for a ribbon graph.

\begin{definition}[\cite{GMT2}]
For $\bullet \in \{\delta, \tau, \delta\tau, \tau\delta, \tau\delta\tau\}$,  the \emph{partial-$\bullet$ polynomial} of a  ribbon graph $G$ is defined as the generating function
$$^{\partial}\varepsilon_{G}^{\bullet}(z):=\sum_{A\subseteq E(G)}z^{\varepsilon(G^{\bullet(A)})}$$
that enumerates all partial-$\bullet$ duals of $G$ by Euler genus.
\end{definition}

A \emph{graft} \((G, L_G)\) consists of a simple graph \(G\) (a graph without loops or multiple edges) together with a vertex subset \(L_G \subseteq V(G)\) \cite{Moffatt15}. Its adjacency matrix \(\mathbf{A}_{(G, L_G)}\) is the matrix over \(\mathrm{GF}(2)\) with rows and columns indexed by \(V(G)\) defined by
\[
\bigl(\mathbf{A}_{(G, L_G)}\bigr)_{uv}=
\begin{cases}
1 & \text{if } u \neq v \text{ and } u \text{ is adjacent to } v,\\[2mm]
1 & \text{if } u = v \text{ and } u \in L_G,\\[2mm]
0 & \text{otherwise.}
\end{cases}
\]

Let $e$ be an edge of a ribbon graph $G$. If $e$ is a loop at the vertex disc $v$ and $e\cup v$ is homeomorphic
to a M\"obius band, then we call $e$ a {\it non-orientable loop}. Otherwise it is said to be an {\it orientable loop}. A \emph{bouquet} is a ribbon graph having exactly one vertex. Two loops in a bouquet are said to be \emph{interlaced} if, when travelling along the boundary of the unique vertex, their ends are encountered in an alternating order. The \emph{intersection graph} \(I(B)\) of a bouquet \(B\) is the simple graph whose vertex set is \(E(B)\), and in which two vertices (corresponding to loops of \(B\)) are adjacent if and only if the corresponding loops are interlaced in \(B\).

Given a bouquet \(B\), let \(L_{I(B)} \subseteq E(B)\) be the set of its non-orientable loops. Then the graft \(\bigl(I(B),\, L_{I(B)}\bigr)\), called the \emph{intersection graft} of \(B\).   Yan and Jin \cite{QYXJ} showed that the partial-$\delta$ polynomial of a bouquet is determined by its intersection graft. This result, proved using topological graph theory, was later extended via delta-matroid \cite{QYJ4} to all five partial-twuality operations: for any \(\bullet \in \{\delta, \tau, \delta\tau, \tau\delta, \tau\delta\tau\}\), the partial-$\bullet$ polynomial of a bouquet depends only on its intersection graft.

\begin{lemma}[\cite{MB}]\label{lem03}
Let \((I(B), L_{I(B)})\) be the intersection graft of a bouquet \(B\). Then
\[
f(B) = \operatorname{corank}\bigl( \mathbf{A}_{(I(B), L_{I(B)})} \bigr) + 1.
\]
\end{lemma}

\begin{proposition}\label{pro21}
Let \((I(B), L_{I(B)})\) be the intersection graft of a bouquet \(B\). Then
\[
\varepsilon(B)= \operatorname{rank}\bigl( \mathbf{A}_{(I(B), L_{I(B)})} \bigr).
\]
\end{proposition}

\begin{proof}
    By Lemma~\ref{lem03},
\[
f(B) = \operatorname{corank}\bigl( \mathbf{A}_{(I(B), L_{I(B)})} \bigr) + 1.
\]
Since \(v(B) = 1\), Euler's formula gives
\[
\varepsilon(B) = 2 + e(B) - 1 - f(B)
              = \operatorname{rank}\bigl( \mathbf{A}_{(I(B), L_{I(B)})} \bigr).
\]
\end{proof}

Since the vertices of \(I(B)\) are in bijection with the edge set \(E(B)\) of a bouquet \(B\), for any \(F \subseteq E(B)\) we also write \(F\) for the corresponding vertex subset of \(I(B)\).

\begin{theorem}\label{lem02}
Let \( B \) be a bouquet with intersection graph \( I(B) \), \( S \) be the set of non-orientable loops of \( B \), and \( F \subseteq E(B) \). Then
\begin{enumerate}
\item[$(1)$] \(\varepsilon(B^{\tau(F)})= \operatorname{rank}\left( \mathbf{A}_{(I(B), S\Delta F)}\right),\)

\item[$(2)$] $\varepsilon(B^{\delta(F)})=\operatorname{rank}(\mathbf A_{(I(B), S)}[F])+\operatorname{rank}(\mathbf A_{(I(B), S)}[F^c]),$

\item[$(3)$] \(\varepsilon(B^{\delta\tau(F)}) = \operatorname{rank}\left( \mathbf{A}_{(I(B), S\Delta F)}[F] \right)+\operatorname{rank}\left( \mathbf{A}_{(I(B), S)}[F^c] \right),\)

\item[$(4)$] \(\varepsilon(B^{\tau\delta(F)}) = \operatorname{rank}\left( \mathbf{A}_{(I(B), S\Delta F)} \right) - \operatorname{corank}\left( \mathbf{A}_{(I(B), S)}[F] \right),\)

\item[$(5)$] \(\varepsilon(B^{\tau\delta\tau(F)}) = \operatorname{rank}\left( \mathbf{A}_{(I(B), S)} \right) - \operatorname{corank}\left( \mathbf{A}_{(I(B), S\Delta F)}[F] \right).\)
\end{enumerate}
\end{theorem}

\begin{proof}
 We first prove (1) separately, as it follows directly from Proposition~\ref{pro21}
\[
\varepsilon(B^{\tau(F)})  = \operatorname{rank}\bigl( \mathbf{A}_{(I(B), S\Delta F)} \bigr).
\]

\begin{figure}
    \centering
\includegraphics[width=0.53\linewidth]{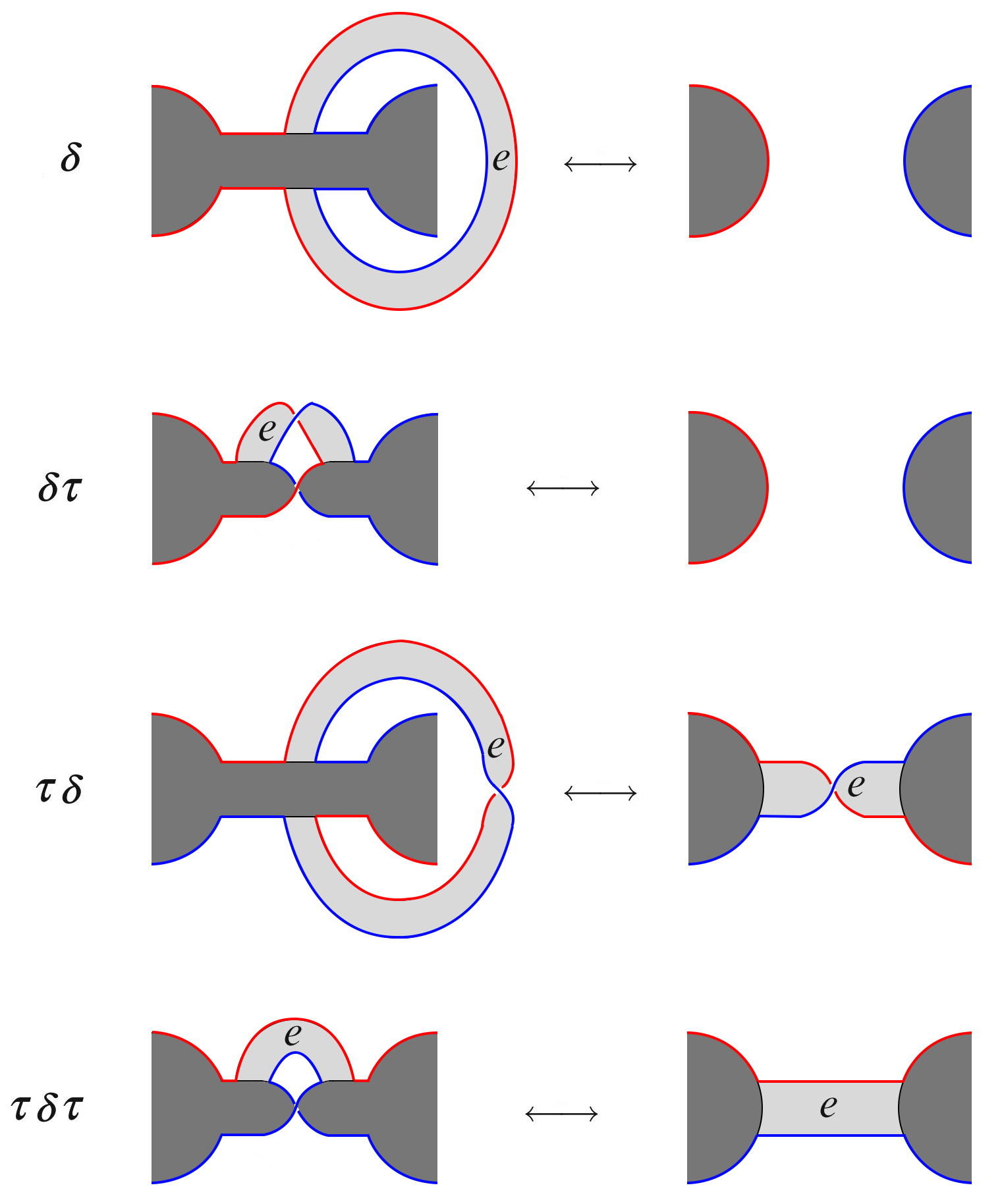}
    \caption{Proof of Theorem \ref{lem02}}
    \label{f_2}
\end{figure}

Set $n = |E(B)|$. For the remaining four statements, the proof follows a common pattern. First, consider the case of a single edge \(e\in E(B)\). By Figure~\ref{f_1}, the operations satisfy
\[
f(B^{\delta(e)})=f(B^{\delta\tau(e)})=f(B \setminus e),\quad
f(B^{\tau\delta(e)})=f(B^{\tau(e)}),\quad
f(B^{\tau\delta\tau(e)})=f(B),
\]
as illustrated in Figure~\ref{f_2} where the transformations of the boundary components under each operation are shown.

For a subset \(F \subseteq E(B)\) and \(\sigma\in \{\delta(F),  \delta\tau(F), \tau\delta(F), \tau\delta\tau(F)\}\), we determine two bouquets \(B_1\) and \(B_2\) such that \(f(B^\sigma) = f(B_1)\) and \(v(B^\sigma) =f((B^\sigma)^*)= f(B_2)\). The bouquet \(B_1\) is obtained by directly extending the single edge identities to the set \(F\). To obtain \(B_2\), we apply Proposition \ref{pro10} to the ribbon graph \((B^\sigma)^*\). Since \((B^\sigma)^* = B^{\sigma\delta(E)}\), we simplify the operator word \(\sigma\delta(E)\) using the group relations to express it in the standard form where the six basic operations act on pairwise disjoint subsets. This yields:

\begin{itemize}
    \item For \(\sigma = \delta(F)\):
          \((B^{\delta(F)})^* = B^{\delta(F)\delta(E)} = B^{\delta(F)\delta(F)\delta(F^c)} = B^{\delta(F^c)}\).

    \item For \(\sigma = \delta\tau(F)\):
          \((B^{\delta\tau(F)})^* = B^{\delta\tau(F)\delta(E)} = B^{\delta\tau(F)\delta(F)\delta(F^c)} = B^{\tau(F)\delta(F^c)}\).

    \item For \(\sigma = \tau\delta(F)\):
          \((B^{\tau\delta(F)})^* = B^{\tau\delta(F)\delta(E)} = B^{\tau\delta(F)\delta(F)\delta(F^c)} = B^{\delta\tau\delta(F)\delta(F^c)}\).

    \item For \(\sigma = \tau\delta\tau(F)\):
          \((B^{\tau\delta\tau(F)})^* = B^{\tau\delta\tau(F)\delta(E)} = B^{\tau\delta\tau(F)\delta(F)\delta(F^c)} = B^{\tau\delta(F)\delta(F^c)}\).
\end{itemize}

Now we determine $B_2$ for each case by applying the single edge identities to the standard forms obtained above.  Since the operations act on disjoint edge sets $F$ and $F^c$, their effects on the number of boundary components are independent and can be composed.

\begin{itemize}
    \item For \(\sigma = \delta(F)\): We have $(B^{\delta(F)})^* = B^{\delta(F^c)}$.
          For a single edge $e$, $f(B^{\delta(e)}) = f(B\setminus e)$.  This identity extends to the whole set $F^c$, giving $f(B^{\delta(F^c)}) = f(B\setminus F^c)$.
          Hence we may take $B_2 = B\setminus F^c$.

    \item For \(\sigma = \delta\tau(F)\): We have $(B^{\delta\tau(F)})^* = B^{\tau(F)\delta(F^c)}$.
          Applying the identity for $\delta$  to the set $F^c$ yields $f(B^{\tau(F)\delta(F^c)}) = f(B^{\tau(F)}\setminus F^c)$.
          Therefore $B_2 = B^{\tau(F)}\setminus F^c$.

    \item For \(\sigma = \tau\delta(F)\):  We have $(B^{\tau\delta(F)})^* = B^{\delta\tau\delta(F)\delta(F^c)}$.
          Using the group relation $\delta\tau\delta = \tau\delta\tau$, this equals $B^{\tau\delta\tau(F)\delta(F^c)}$.
          For a single edge $e$, $f(B^{\tau\delta\tau(e)}) = f(B)$; extending to the set $F$ gives $f(B^{\tau\delta\tau(F)}) = f(B)$.
          Consequently, $f(B^{\tau\delta\tau(F)\delta(F^c)}) = f(B\setminus F^c)$, and we set $B_2 = B\setminus F^c$.

    \item For \(\sigma = \tau\delta\tau(F)\): We have $(B^{\tau\delta\tau(F)})^* = B^{\tau\delta(F)\delta(F^c)}$.
          For a single edge $e$, $f(B^{\tau\delta(e)}) = f(B^{\tau(e)})$; hence for the set $F$ we have $f(B^{\tau\delta(F)}) = f(B^{\tau(F)})$.
          It follows that $f(B^{\tau\delta(F)\delta(F^c)}) = f(B^{\tau(F)}\setminus F^c)$, so $B_2 = B^{\tau(F)}\setminus F^c$.
\end{itemize}
The results are summarized in the following table:

\[
\begin{array}{c|c|c}
\sigma & B_1 & B_2 \\ \hline
\delta(F)      & B \setminus F          & B \setminus F^c \\
\delta\tau(F)  & B \setminus F          & B^{\tau(F)} \setminus F^c \\
\tau\delta(F)  & B^{\tau(F)}    & B \setminus F^c \\
\tau\delta\tau(F) & B        & B^{\tau(F)} \setminus F^c.
\end{array}
\]
Applying Lemma~\ref{lem03} to $B_1$ and $B_2$ yields
\[
f(B^\sigma) = \operatorname{corank}\bigl( \mathbf{A}_{(I(B), T_1)}[J_1] \bigr) + 1, \qquad
v(B^\sigma) = \operatorname{corank}\bigl( \mathbf{A}_{(I(B), T_2)}[J_2] \bigr) + 1,
\]
where the pairs $(T_1,J_1)$ and $(T_2,J_2)$ are given by
\[
\begin{array}{c|cc|cc}
\sigma & T_1 & J_1 & T_2 & J_2 \\ \hline
\delta(F)      & S & F^c        & S & F \\
\delta\tau(F)  & S & F^c        & S\Delta F & F \\
\tau\delta(F)  & S\Delta F & E(B) & S & F \\
\tau\delta\tau(F) & S & E(B) & S\Delta F & F.
\end{array}
\]
Recall that $\mathbf{A}_{(I(B),T_1)}[E(B)] = \mathbf{A}_{(I(B),T_1)}$.
Since $e(B^\sigma)=n$, Euler's formula gives
\[
\begin{aligned}
\varepsilon(B^\sigma) &= 2+n-v(B^\sigma)-f(B^\sigma)\\
&= n - \operatorname{corank}\bigl( \mathbf{A}_{(I(B), T_2)}[J_2] \bigr) - \operatorname{corank}\bigl( \mathbf{A}_{(I(B), T_1)}[J_1] \bigr).
\end{aligned}
\]

\item [{\bf (2)}]   For $\sigma=\delta(F)$,
\[
\begin{aligned}
\varepsilon(B^{\delta(F)})
&= n - \operatorname{corank}(\mathbf{A}_{(I(B), S)}[F]) - \operatorname{corank}(\mathbf{A}_{(I(B), S)}[F^c]) \\
&= \operatorname{rank}(\mathbf{A}_{(I(B), S)}[F]) + \operatorname{rank}(\mathbf A_{(I(B), S)}[F^c]).
\end{aligned}
\]

\item [{\bf (3)}]   For $\sigma=\delta\tau(F)$,
\[
\begin{aligned}
\varepsilon(B^{\delta\tau(F)})
&= n - \operatorname{corank}(\mathbf{A}_{(I(B), S\Delta F)}[F]) - \operatorname{corank}(\mathbf{A}_{(I(B), S)}[F^c]) \\
&= \operatorname{rank}(\mathbf{A}_{(I(B), S\Delta F)}[F])+\operatorname{rank}(\mathbf A_{(I(B), S)}[F^c]).
\end{aligned}
\]

\item [{\bf (4)}]   For $\sigma=\tau\delta(F)$,
\[
\begin{aligned}
\varepsilon(B^{\tau\delta(F)})
&= n  - \operatorname{corank}(\mathbf{A}_{(I(B), S)}[F])- \operatorname{corank}(\mathbf{A}_{(I(B), S\Delta F)}) \\
&= \operatorname{rank}(\mathbf{A}_{(I(B), S\Delta F)}) - \operatorname{corank}(\mathbf A_{(I(B), S)}[F]).
\end{aligned}
\]

\item [{\bf (5)}]   For $\sigma=\tau\delta\tau(F)$,
\[
\begin{aligned}
\varepsilon(B^{\tau\delta\tau(F)})
&= n  - \operatorname{corank}(\mathbf{A}_{(I(B), S\Delta F)}[F])- \operatorname{corank}(\mathbf{A}_{(I(B), S)}) \\
&= \operatorname{rank}(\mathbf{A}_{(I(B), S)}) - \operatorname{corank}(\mathbf A_{(I(B), S\Delta F)}[F]).
\end{aligned}
\]
This completes the proof of all five statements.
\end{proof}

\subsection{Partial-twuality polynomials of matrices}

Theorem~\ref{lem02} expresses the Euler genus of a bouquet under a partial-twuality operation as a linear combination of ranks and/or coranks of  adjacency matrices of related intersection grafts. For a graft $(G, L_G)$, its adjacency matrix over $\mathrm{GF}(2)$ satisfies the following relation for any subset $A \subseteq V(G)$:
\[
\mathbf{A}_{(G,\,L_G \Delta A)} = \mathbf{A}_{(G,L_G)} + I_A,
\]
where $I_A$ is the diagonal matrix with $(I_A)_{ii}=1$ for $i \in A$ and $0$ otherwise. Therefore, the five Euler genus formulas can be rewritten solely in terms of ranks and/or coranks of matrices and of matrices modified by $I_A$. These new expressions are formulated entirely within the language of matrix operations and are thus well-defined for any square matrix over any field.

\begin{definition}\label{maindef}
Let \(M\) be a square matrix over a field, with rows and columns indexed by a finite set \(V\). For $\bullet \in \{ \delta, \tau, \delta\tau, \tau\delta, \tau\delta\tau\}$, the \emph{partial-$\langle\bullet\rangle$ polynomial} of \(M\) is defined as
\[
P_{\langle \bullet \rangle}(M, z) = \sum_{A \subseteq V} z^{\,r_{\langle \bullet \rangle}(M, A)},
\]
where the exponent function \(r_{\langle \bullet \rangle}(M, A)\) is given by
\[
\begin{aligned}
r_{\langle \delta \rangle}(M, A) &= \operatorname{rank}\!\bigl(M[A]\bigr) + \operatorname{rank}\!\bigl(M[A^c]\bigr),\\
r_{\langle \tau \rangle}(M, A) &= \operatorname{rank}(M + I_A),\\
r_{\langle \delta\tau \rangle}(M, A) &= \operatorname{rank}\!\bigl((M + I_A)[A]\bigr) + \operatorname{rank}\!\bigl(M[A^c]\bigr),\\
r_{\langle \tau\delta \rangle}(M, A) &= \operatorname{rank}(M + I_A) - \operatorname{corank}\!\bigl(M[A]\bigr),\\
r_{\langle \tau\delta\tau \rangle}(M, A) &= \operatorname{rank}(M) - \operatorname{corank}\!\bigl((M + I_A)[A]\bigr).
\end{aligned}
\]
Here \(M[A]\) denotes the principal submatrix of \(M\) on \(A \subseteq V\), and \(A^c = V \setminus A\). In particular, if $M$ is the $0 \times 0$ matrix (the empty matrix), then $P_{\langle\bullet\rangle}(M,z)=1$.
\end{definition}

\begin{definition}\label{graph-polynomial-def}
For $\bullet \in \{ \delta, \tau, \delta\tau, \tau\delta, \tau\delta\tau\}$, the \emph{partial-$\langle \bullet \rangle$ polynomial} of a graft $(G,L_G)$ is
\[
P_{(G,L_G)}^{\langle \bullet \rangle}(z) := P_{\langle \bullet \rangle}(\mathbf A_{(G,L_G)}, z).
\]
\end{definition}

Substituting $\mathbf A_{(G,L_G)} + I_A = \mathbf A_{(G, L_G \Delta A)}$ into Definition~\ref{maindef} with $M = \mathbf A_{(G,L_G)}$ yields the following expansions for the partial-$\langle \bullet \rangle$ polynomials of a graft:

\begin{proposition}\label{pro05}
For $\bullet \in \left\{\delta, \tau, \delta\tau, \tau\delta, \tau\delta\tau\right\}$, the partial-$\langle \bullet \rangle$ polynomial of a graft $(G,L_G)$ is
\begin{enumerate}
\item[$(1)$] $\displaystyle
P_{(G,L_G)}^{\langle \delta \rangle}(z) = \sum_{A \subseteq V(G)}
z^{\operatorname{rank}(\mathbf A_{(G,L_G)}[A])+\operatorname{rank}(\mathbf A_{(G,L_G)}[A^c])},$

\item[$(2)$] $\displaystyle
P_{(G,L_G)}^{\langle \tau \rangle}(z) = \sum_{A \subseteq V(G)}
z^{\operatorname{rank}(\mathbf A_{(G,L_G\Delta A)})},$

\item[$(3)$] $\displaystyle
P_{(G,L_G)}^{\langle \delta\tau \rangle}(z) = \sum_{A \subseteq V(G)}
z^{\operatorname{rank}(\mathbf A_{(G,L_G\Delta A)}[A])+\operatorname{rank}(\mathbf A_{(G,L_G)}[A^c])},$

\item[$(4)$] $\displaystyle
P_{(G,L_G)}^{\langle \tau\delta \rangle}(z) = \sum_{A \subseteq V(G)}
z^{\operatorname{rank}(\mathbf A_{(G,L_G\Delta A)})-\operatorname{corank}(\mathbf A_{(G,L_G)}[A])},$

\item[$(5)$] $\displaystyle
P_{(G,L_G)}^{\langle \tau\delta\tau \rangle}(z) = \sum_{A \subseteq V(G)}
z^{\operatorname{rank}(\mathbf A_{(G,L_G)})-\operatorname{corank}(\mathbf A_{(G,L_G\Delta A)}[A])}.$
\end{enumerate}
\end{proposition}

\begin{remark}\label{remark1}
{\bf(1)}
For any graft $(G,L_G)$, the partial-$\langle \tau \rangle$ polynomial satisfies
\[
P_{(G,L_G)}^{\langle \tau \rangle}(z) = P_{(G,\emptyset)}^{\langle \tau \rangle}(z),
\]
since the change of summation variable $B = L_G \mathbin{\Delta} A$ yields
\[
P_{(G,L_G)}^{\langle \tau \rangle}(z) = \sum_{A \subseteq V(G)} z^{\operatorname{rank}(\mathbf A_{(G,L_G\Delta A)})}
= \sum_{B \subseteq V(G)} z^{\operatorname{rank}(\mathbf A_{(G,B)})}
= P_{(G,\emptyset)}^{\langle \tau \rangle}(z).
\]
Thus the polynomial depends only on the underlying graph $G$, not on the subset $L_G$.

\noindent{\bf(2)}~ The partial-$\langle \tau\delta\tau \rangle$ polynomial is closely related to the interlace polynomial of a graft \cite{Arratia2004}.
For a graft $(G,L_G)$, the interlace polynomial is defined as
\[
q((G,L_G),x)=\sum_{U \subseteq V(G)} (x-1)^{\operatorname{corank}(\mathbf A_{(G,L_G)}[U])}.
\]
Then one can verify that
\[
P_{(G,L_G)}^{\langle \tau\delta\tau \rangle}(z)
= z^{\operatorname{rank}(\mathbf A_{(G,L_G)})}\,
  q\!\left((G,\,V(G)\setminus L_G),\ 1+\frac{1}{z}\right).
\]
\end{remark}

\begin{theorem}\label{thm:main-equivalence}
Let \((I(B), L_{I(B)})\) be the intersection graft of a bouquet \(B\).
Then for $\bullet \in \{\delta, \tau, \delta\tau, \tau\delta, \tau\delta\tau\}$,
\[
{}^{\partial}\varepsilon^{\bullet}_{B}(z) \;=\; P^{\langle \bullet \rangle}_{(I(B), L_{I(B)})}(z).\]
\end{theorem}

\begin{proof}
For each $\bullet$ and each $A\subseteq E(B)$, the exponent $\varepsilon(B^{\bullet(A)})$ in the definition of $^{\partial}\varepsilon^{\bullet}_{B}(z)$ is given by Theorem~\ref{lem02}, which is precisely the exponent in the corresponding term of $P^{\langle\bullet\rangle}_{(I(B), L_{I(B)})}(z)$ as shown in Proposition~\ref{pro05}. Hence the two sums are equal.
\end{proof}

\begin{remark}
When $\bullet = \delta$ and the bouquet $B$ is orientable (so that $L_{I(B)} = \emptyset$), Theorem~\ref{thm:main-equivalence} specializes to
\[
{}^{\partial}\varepsilon^{\delta}_{B}(z) \;=\; P^{\langle \delta \rangle}_{(I(B),\,\emptyset)}(z),
\]
which is precisely the identity established by Cheng~\cite{Cheng} for orientable bouquets.  Starting from this case, Cheng \cite{Cheng} defined the partial-$\langle \delta \rangle$ polynomial for simple graphs. Our Definition~\ref{graph-polynomial-def} thus encompasses Cheng's definition and further extends it to the other four twuality operations.
\end{remark}

\subsection{Examples}

To illustrate the definitions, we compute the partial-$\langle\tau\delta\tau\rangle$, partial-$\langle\delta\tau\rangle$ and partial-$\langle\tau\delta\rangle$ polynomials for the complete graph $K_n$ regarded as the graft $(K_n, \emptyset)$. The partial-$\langle\delta\rangle$ polynomial for this graft was computed by Cheng~\cite{Cheng} and the partial-$\langle\tau\rangle$ polynomial for this graft was computed by Yan and Li~\cite{QYL}.

\begin{example}\label{ex:Kn}
Consider the complete graph $K_n$ on $n$ vertices.
Let $M = \mathbf A_{(K_n,\emptyset)}$ be its adjacency matrix (all off-diagonal entries are $1$, diagonal entries are $0$).
It is known that
\[
\operatorname{rank}(M)= \begin{cases}
n, & \text{if } n \text{ is even},\\[2mm]
n-1, & \text{if } n \text{ is odd}.
\end{cases}
\]
For any subset $A\subseteq V(K_n)$, the principal submatrix $(M+I_A)[A]$ coincides with the adjacency matrix of the graft $(K_{|A|},A)$, i.e.,
$\mathbf A_{(K_{|A|},A)}$.
For $A\neq\emptyset$, $\operatorname{rank}\bigl(\mathbf A_{(K_{|A|},A)}\bigr)=1$; for $A=\emptyset$ the rank is $0$.
Note that \[
\operatorname{rank}(\mathbf A_{(K_n, A)}) =
\begin{cases}
n, & A=\emptyset,\ n\ \text{even},\\
n-1, & A=\emptyset,\ n\ \text{odd},\\
n-|A|+1, & A\neq\emptyset.
\end{cases}
\]

\begin{itemize}
  \item \textbf{Partial-$\langle\tau\delta\tau\rangle$ polynomial.}
    \begin{itemize}
      \item If $n$ is even,
        \begin{align*}
          P_{(K_n,\emptyset)}^{\langle\tau\delta\tau\rangle}(z)
          &=\sum_{A\subseteq V(K_n)} z^{\,n-|A|+\operatorname{rank}(\mathbf A_{(K_{|A|},A)})} \\
          &=z^{n} +\sum_{\substack{A\subseteq V(K_n)\\A\neq\emptyset}}\binom{n}{|A|}z^{\,n+1-|A|} \\
          &=z(1+z)^{n}+z^{n}-z^{n+1}.
        \end{align*}
      \item If $n$ is odd,
        \begin{align*}
          P_{(K_n,\emptyset)}^{\langle\tau\delta\tau\rangle}(z)
          &=\sum_{A\subseteq V(K_n)} z^{\,n-1-|A|+\operatorname{rank}(\mathbf A_{(K_{|A|},A)})} \\
          &=z^{n-1} +\sum_{\substack{A\subseteq V(K_n)\\A\neq\emptyset}}\binom{n}{|A|}z^{\,n-|A|} \\
          &=(1+z)^{n}+z^{n-1}-z^{n}.
        \end{align*}
    \end{itemize}

  \item \textbf{Partial-$\langle\delta\tau\rangle$ polynomial.}
    Recall that
    \[
    r_{\langle\delta\tau\rangle}(M,A)=\operatorname{rank}\!\bigl(\mathbf A_{(K_{|A|},A)}\bigr)
      +\operatorname{rank}\!\bigl(\mathbf A_{(K_{|A^{c}|},\emptyset)}\bigr).
    \]
    The matrix $\mathbf A_{(K_{|A^{c}|},\emptyset)}$ has zero diagonal and all one off-diagonal; its rank is $|A^{c}|$ when $|A^{c}|$ is even and $|A^{c}|-1$ when $|A^{c}|$ is odd.
    \begin{itemize}
      \item If $n$ is even,
        \begin{align*}
          P_{(K_n,\emptyset)}^{\langle\delta\tau\rangle}(z)
          &=z^{n}
            +\sum_{\substack{A\neq\emptyset\\|A|\text{ even}}}\binom{n}{|A|}z^{1+|A^{c}|}
            +\sum_{\substack{A\neq\emptyset\\|A|\text{ odd}}}\binom{n}{|A|}z^{|A^{c}|} \\[1mm]
          &=z^{n}+z\frac{(1+z)^{n}+(1-z)^{n}}{2}-z^{n+1}
            +\frac{(1+z)^{n}-(1-z)^{n}}{2} \\[1mm]
          &=z^{n}+\frac{(1+z)^{\,n+1}-(1-z)^{\,n+1}}{2}-z^{n+1}.
        \end{align*}
      \item If $n$ is odd,
        \begin{align*}
          P_{(K_n,\emptyset)}^{\langle\delta\tau\rangle}(z)
          &=z^{n-1}
            +\sum_{\substack{A\neq\emptyset\\|A|\text{ odd}}}\binom{n}{|A|}z^{1+|A^{c}|}
            +\sum_{\substack{A\neq\emptyset\\|A|\text{ even}}}\binom{n}{|A|}z^{|A^{c}|} \\[1mm]
          &=z^{n-1}+z\frac{(1+z)^{n}+(1-z)^{n}}{2}
            +\frac{(1+z)^{n}-(1-z)^{n}}{2}-z^{n} \\[1mm]
          &=z^{n-1}+\frac{(1+z)^{\,n+1}-(1-z)^{\,n+1}}{2}-z^{n}.
        \end{align*}
    \end{itemize}

   \item \textbf{Partial-$\langle\tau\delta\rangle$ polynomial.}
  For a subset $A\subseteq V(K_n)$, we have
  \[
  r_{\langle\tau\delta\rangle}(M,A)=\operatorname{rank}\!\bigl(\mathbf A_{(K_n,A)}\bigr)-\operatorname{corank}\!\bigl(\mathbf A_{(K_n,\emptyset)}[A]\bigr).
  \]
  Thus
  \[
  r_{\langle\tau\delta\rangle}(M,A)=\begin{cases}
      n,        & A=\emptyset,\ n\ \text{even},\\[1mm]
      n-1,      & A=\emptyset,\ n\ \text{odd},\\[1mm]
      n-|A|+1,    & A\neq\emptyset,\ |A|\ \text{even},\\[1mm]
      n-|A|,      & A\neq\emptyset,\ |A|\ \text{odd}.
    \end{cases}
  \]

  \begin{itemize}
    \item If $n$ is even,
      \begin{align*}
        P_{(K_n,\emptyset)}^{\langle\tau\delta\rangle}(z)
        &=z^{n}
          +\sum_{\substack{A\neq\emptyset\\ |A|\ \text{even}}}\binom{n}{|A|}z^{\,n-|A|+1}
          +\sum_{\substack{A\neq\emptyset\\ |A|\ \text{odd}}}\binom{n}{|A|}z^{\,n-|A|} \\[1mm]
        &=z^{n}+z\frac{(1+z)^{n}+(1-z)^{n}}{2}-z^{\,n+1}
          +\frac{(1+z)^{n}-(1-z)^{n}}{2} \\[1mm]
        &=z^{n}+\frac{(1+z)^{\,n+1}-(1-z)^{\,n+1}}{2}-z^{\,n+1}.
      \end{align*}
    \item If $n$ is odd,
      \begin{align*}
        P_{(K_n,\emptyset)}^{\langle\tau\delta\rangle}(z)
        &=z^{n-1}
          +\sum_{\substack{A\neq\emptyset\\ |A|\ \text{even}}}\binom{n}{|A|}z^{\,n-|A|+1}
          +\sum_{\substack{A\neq\emptyset\\ |A|\ \text{odd}}}\binom{n}{|A|}z^{\,n-|A|} \\[1mm]
        &=z^{n-1}+z\frac{(1+z)^{n}+(1-z)^{n}}{2}
          +\frac{(1+z)^{n}-(1-z)^{n}}{2}-z^{\,n+1} \\[1mm]
        &=z^{n-1}+\frac{(1+z)^{\,n+1}+(1-z)^{\,n+1}}{2}-z^{\,n+1}.
      \end{align*}
  \end{itemize}


\end{itemize}
\end{example}

\section{Basic properties of partial-twuality polynomials}

\subsection{Product formulas and reduction for isolated vertices}

\begin{proposition}\label{prop:product-disjoint-union}
Let $\bullet \in \{\delta, \tau, \delta\tau, \tau\delta, \tau\delta\tau\}$, and let $M$ be a square matrix over a field with rows and columns indexed by a finite set $V$.  If $V$ is the disjoint union of two subsets $V_1$ and $V_2$ such that $M$ is block-diagonal with respect to this partition, i.e., $M = M[V_1] \oplus M[V_2]$, then
\[
P_{\langle \bullet \rangle}(M, z) = P_{\langle \bullet \rangle}\bigl(M[V_1], z\bigr) P_{\langle \bullet \rangle}\bigl(M[V_2], z\bigr).
\]
\end{proposition}

\begin{proof}
For any subset $A \subseteq V$, denote $A_1 = A \cap V_1$ and $A_2 = A \cap V_2$.
Since $M$ is block-diagonal with respect to the partition $V = V_1 \sqcup V_2$, every matrix that appears in Definition~\ref{maindef} (for instance, $M[A]$, $M+I_A$, $(M+I_A)[A]$ and $M[A^c]$) is itself block-diagonal with blocks determined only by $M[V_1]$, $M[V_2]$, $A_1$, and $A_2$.
Consequently, the rank (resp. corank) of each such matrix is the sum of the ranks (resp. coranks) of its two blocks.
Hence, for each exponent function $r_{\langle\bullet\rangle}$ in Definition~\ref{maindef} we have
\[
r_{\langle\bullet\rangle}(M, A) = r_{\langle\bullet\rangle}\bigl(M[V_1], A_1\bigr) + r_{\langle\bullet\rangle}\bigl(M[V_2], A_2\bigr).
\]
Therefore,
\[
\begin{aligned}
P_{\langle\bullet\rangle}(M,z)
&= \sum_{A\subseteq V} z^{r_{\langle\bullet\rangle}(M,A)} \\
&= \sum_{A_1\subseteq V_1} \sum_{A_2\subseteq V_2} z^{r_{\langle\bullet\rangle}(M[V_1],A_1) + r_{\langle\bullet\rangle}(M[V_2],A_2)} \\
&= \Bigl(\sum_{A_1\subseteq V_1} z^{r_{\langle\bullet\rangle}(M[V_1],A_1)}\Bigr) \Bigl(\sum_{A_2\subseteq V_2} z^{r_{\langle\bullet\rangle}(M[V_2],A_2)}\Bigr) \\
&= P_{\langle\bullet\rangle}\bigl(M[V_1], z\bigr) P_{\langle\bullet\rangle}\bigl(M[V_2], z\bigr).
\end{aligned}
\]
\end{proof}

The multiplicative formula of Proposition~\ref{prop:product-disjoint-union} applies directly to grafts.

\begin{corollary}
Let $\bullet \in \{\delta, \tau, \delta\tau, \tau\delta, \tau\delta\tau\}$.
If a graft $(G, L_G)$ is the disjoint union of two grafts $(G_1, L_{G_1})$ and $(G_2, L_{G_2})$, then
\[
P_{(G, L_G)}^{\langle \bullet \rangle}(z) = P_{(G_1, L_{G_1})}^{\langle \bullet \rangle}(z)  P_{(G_2, L_{G_2})}^{\langle \bullet \rangle}(z).
\]
\end{corollary}

We now focus on the behaviour of the partial-$\langle\bullet\rangle$ polynomials in the presence of an isolated vertex.
Let $M$ be a square matrix over a field with rows and columns indexed by a finite set $V$. A vertex $v \in V$ is called \emph{isolated} if $M_{uv}=M_{vu}=0$ for all $u \in V \setminus \{v\}$.
The following proposition, stated over $\mathrm{GF}(2)$, shows that an isolated vertex contributes a multiplicative factor depending only on its diagonal entry $M_{vv}$, thereby allowing a reduction.

\begin{proposition}\label{prop:isolated-vertex-matrix}
Let $M$ be a square matrix over $\mathrm{GF}(2)$, with rows and columns indexed by a finite set $V$, and let $v \in V$ be an isolated vertex.
Let $M'=M[V\setminus\{v\}]$.
Then for $\bullet\in\{\delta,\tau,\delta\tau,\tau\delta,\tau\delta\tau\}$ the partial-$\langle\bullet\rangle$ polynomial  of $M$ satisfies
\[
P_{\langle\bullet\rangle}(M,z)=c_{\bullet}(M_{vv})\,P_{\langle\bullet\rangle}(M',z),
\]
where the factor $c_{\bullet}(M_{vv})$ is given in the following table:
\[
\begin{array}{c|c|c}
\bullet & M_{vv}=0 & M_{vv}=1\\ \hline
\delta & 2 & 2z\\
\tau & z+1 & z+1\\
\delta\tau & z+1 & z+1\\
\tau\delta & 2 & z+1\\
\tau\delta\tau & 2 & z+1
\end{array}
\]
\end{proposition}

\begin{proof}
Denote $V' = V \setminus \{v\}$. Since $M_{uv}=M_{vu}=0$ for all $u \in V \setminus \{v\}$, the matrix $M$ is block-diagonal with respect to the partition $\{v\} \cup V'$, i.e., $M = [M_{vv}] \oplus M'$.  By Proposition~\ref{prop:product-disjoint-union} we have
\[
P_{\langle\bullet\rangle}(M,z)=P_{\langle\bullet\rangle}([M_{vv}],z) P_{\langle\bullet\rangle}(M',z).
\]
Thus it suffices to compute the five polynomials for the $1\times1$ matrix $[M_{vv}]$.
For a set $A\subseteq\{v\}$ there are two possibilities: $A=\varnothing$ and $A=\{v\}$.

A direct calculation using Definition~\ref{maindef} yields the following values of the exponent $r_{\langle\bullet\rangle}([M_{vv}],A)$:

\[
\begin{array}{c|c|c}
\bullet & A=\varnothing & A=\{v\}\\ \hline
\delta & M_{vv} & M_{vv}\\
\tau & M_{vv} & 1-M_{vv}\\
\delta\tau & M_{vv} & 1-M_{vv}\\
\tau\delta & M_{vv} & 0\\
\tau\delta\tau & M_{vv} & 0
\end{array}
\]
Hence
\[
\begin{aligned}
P_{\langle\delta\rangle}([M_{vv}],z)&=z^{M_{vv}}+z^{M_{vv}}=2z^{M_{vv}},\\
P_{\langle\tau\rangle}([M_{vv}],z)&=z^{M_{vv}}+z^{1-M_{vv}}=z+1,\\
P_{\langle\delta\tau\rangle}([M_{vv}],z)&=z^{M_{vv}}+z^{1-M_{vv}}=z+1,\\
P_{\langle\tau\delta\rangle}([M_{vv}],z)&=z^{M_{vv}}+z^{0}=1+z^{M_{vv}},\\
P_{\langle\tau\delta\tau\rangle}([M_{vv}],z)&=z^{M_{vv}}+z^{0}=1+z^{M_{vv}}.
\end{aligned}
\]
Substituting $M_{vv}=0$ and $M_{vv}=1$ gives the factors listed in the table.
\end{proof}

\subsection {Leaf-reduction formulas for grafts}

\begin{proposition}[\cite{Cheng}]\label{prop:delta-simple}
Let $G$ be a simple graph and $x, y\in V(G)$ be two adjacent vertices.
If $x$ is a leaf, i.e., the degree of $x$ is one, then
\[
P_{(G,\emptyset)}^{\langle \delta \rangle}(z)
= P_{(G \setminus \{x\},\emptyset)}^{\langle \delta \rangle}(z) + 2z^{2}P_{(G \setminus \{x,y\},\emptyset)}^{\langle \delta \rangle}(z).
\]
\end{proposition}

Li \cite{Li2025} extended Proposition~\ref{prop:delta-simple} to grafts as follows.

\begin{proposition}[\cite{Li2025}]
Let $(G,L_G)$ be a graft and let $x,y\in V(G)$ be two adjacent vertices such that  $x$ is a leaf and $x\notin L_G$. Then
\[
P_{(G,L_G)}^{\langle\delta\rangle}(z)
   = P_{(G \setminus \{x\},\;L_G)}^{\langle\delta\rangle}(z)
   \;+\;2z^{2}\,P_{(G \setminus \{x,y\},\;L_G\setminus\{y\})}^{\langle\delta\rangle}(z).
\]
\end{proposition}

We obtain the following leaf-reduction for the partial-$\langle\tau\rangle$ polynomial.

\begin{proposition}
Let $(G, L_G)$ be a graft and let $x, y \in V(G)$ be adjacent vertices. If  $x$ is a leaf, then
\[
P_{(G, L_G)}^{\langle \tau \rangle}(z)
= z \, P_{(G \setminus \{x\},\;L_G\setminus \{x\})}^{\langle \tau \rangle}(z) + 2z^{2} \, P_{(G \setminus \{x,y\},\;L_G\setminus \{x,y\})}^{\langle \tau \rangle}(z).
\]
\end{proposition}

\begin{proof}
Remark~\ref{remark1} shows that the partial-$\langle\tau\rangle$ polynomial is independent of $L_G$. Consequently, it suffices to establish the formula for $(G, \emptyset)$.
We divide all the subsets of the vertex set of $G$ into four parts $V_1, V_2, V_3, V_4$, where
\begin{itemize}
  \item[~] $V_1=\{A\subseteq V(G)\mid x\notin A,\ y\notin A\}$,
  \item[~] $V_2=\{A\subseteq V(G)\mid x\notin A,\ y\in A\}$,
  \item[~] $V_3=\{A\subseteq V(G)\mid x\in A,\ y\in A\}$,
  \item[~] $V_4=\{A\subseteq V(G)\mid x\in A,\ y\notin A\}$.

\end{itemize}
According to Proposition \ref{pro05}, we have

\begin{itemize}
\item For $A\in V_1$,
  \begin{align*}
    \operatorname{rank}(\mathbf{A}_{(G, A)})
    &=\operatorname{rank}\begin{pmatrix}
        0 & 1 & \mathbf{0} \\
        1 & 0 & \ast \\
        \mathbf{0} & \ast & \ast
      \end{pmatrix} \\
    &=\operatorname{rank}\begin{pmatrix}
        0 & 1 & \mathbf{0} \\
        1 & 0 & \mathbf{0} \\
        \mathbf{0} & \mathbf{0} & \ast
      \end{pmatrix}
    =\operatorname{rank}(\mathbf{A}_{(G\setminus \{x,y\}, A)})+2.
  \end{align*}
  Hence,
  \[
    \sum_{A\in V_1}z^{\operatorname{rank}(\mathbf{A}_{(G, A)})}=z^2P_{(G\setminus \{x,y\},\emptyset)}^{\langle \tau \rangle}(z).
  \]

\item For $A\in V_2$,
  \begin{align*}
    \operatorname{rank}(\mathbf{A}_{(G, A)})
    &=\operatorname{rank}\begin{pmatrix}
        0 & 1 & \mathbf{0} \\
        1 & 1 & \ast \\
        \mathbf{0} & \ast & \ast
      \end{pmatrix} \\
    &=\operatorname{rank}\begin{pmatrix}
        0 & 1 & \mathbf{0} \\
        1 & 0 & \mathbf{0} \\
        \mathbf{0} & \mathbf{0} & \ast
      \end{pmatrix}
    =\operatorname{rank}(\mathbf{A}_{(G\setminus \{x,y\}, A\setminus\{y\})})+2.
  \end{align*}
  Hence,
  \[
    \sum_{A\in V_2}z^{\operatorname{rank}(\mathbf{A}_{(G, A)})}=z^2P_{(G\setminus \{x,y\},\emptyset)}^{\langle \tau \rangle}(z).
  \]

\item For $A\in V_3$,
  \begin{align*}
    \operatorname{rank}(\mathbf{A}_{(G, A)})
    &=\operatorname{rank}\begin{pmatrix}
        1 & 1 & \mathbf{0} \\
        1 & 1 & \ast \\
        \mathbf{0} & \ast & \ast
      \end{pmatrix} \\
    &=\operatorname{rank}\begin{pmatrix}
        1 & 0 & \mathbf{0} \\
        0 & 0 & \ast \\
        \mathbf{0} & \ast & \ast
      \end{pmatrix}
    =\operatorname{rank}(\mathbf{A}_{(G\setminus\{x\}, A\setminus\{x,y\})})+1.
  \end{align*}

  \item For $A\in V_4$,
  \begin{align*}
    \operatorname{rank}(\mathbf{A}_{(G, A)})
    &=\operatorname{rank}\begin{pmatrix}
        1 & 1 & \mathbf{0} \\
        1 & 0 & \ast \\
        \mathbf{0} & \ast & \ast
      \end{pmatrix} \\
    &=\operatorname{rank}\begin{pmatrix}
        1 & 0 & \mathbf{0} \\
        0 & 1 & \ast \\
        \mathbf{0} & \ast & \ast
      \end{pmatrix}
    =\operatorname{rank}(\mathbf{A}_{(G\setminus\{x\}, (A\setminus\{x\})\cup\{y\})})+1.
  \end{align*}
\end{itemize}

Now note that the sets $V_3$ and $V_4$ together consist of all subsets of $V(G)$ that contain $x$. Moreover, the map
\[
\varphi: V_3 \cup V_4 \longrightarrow \mathcal{P}(V(G\setminus\{x\})), \qquad
\varphi(A)=\begin{cases}
A\setminus\{x,y\}, & \text{if } A \in V_3,\\
(A\setminus\{x\})\cup\{y\}, & \text{if } A \in V_4,
\end{cases}
\]
is a bijection. For $A \in V_3 \cup V_4$, we have
\[
\operatorname{rank}(\mathbf{A}_{(G, A)}) = \operatorname{rank}(\mathbf{A}_{(G\setminus\{x\}, \varphi(A))}) + 1.
\]
Therefore,
\begin{align*}
\sum_{A\in V_3\cup V_4} z^{\operatorname{rank}(\mathbf{A}_{(G, A)})}
&= z \sum_{\varphi(A)\subseteq V(G\setminus\{x\})} z^{\operatorname{rank}(\mathbf{A}_{(G\setminus\{x\}, \varphi(A))})} \\
&= z P_{(G\setminus\{x\},\emptyset)}^{\langle \tau \rangle}(z).
\end{align*}

Combining all four parts, we obtain
\begin{align*}
P_{(G,\emptyset)}^{\langle \tau \rangle}(z)
&= \sum_{i=1}^{4}\sum_{A\in V_i}z^{\operatorname{rank}(\mathbf{A}_{(G, A)})} \\
&= z P_{(G\setminus\{x\},\emptyset)}^{\langle \tau \rangle}(z) + (2z^{2}) P_{(G\setminus\{x, y\},\emptyset)}^{\langle \tau \rangle}(z).
\end{align*}
\end{proof}

\subsection{Bounds on the degree of partial-$\langle\bullet\rangle$ polynomials}

For a nonzero polynomial \(P(z) = \sum_i c_i z^i\), we denote by \(\min\deg P(z)\) the smallest exponent \(i\) with \(c_i \neq 0\), and by \(\deg P(z)\) the largest such index. The following proposition shows that  $P_{\langle \bullet \rangle}(M,z)$ is a polynomial.

\begin{proposition}\label{mindegree}
    Let \(M\) be a square matrix over a field, with rows and columns indexed by a finite set \(V\). Then the partial-$\langle\bullet\rangle$ polynomials satisfy
    \[
    \min\deg P_{\langle \bullet \rangle}(M,z) \geq 0 \qquad\text{for } \bullet \in \{ \delta, \tau, \delta\tau, \tau\delta, \tau\delta\tau\}.
    \]
\end{proposition}

\begin{proof}
For \(\bullet \in \{\delta, \tau, \delta\tau\}\), this is immediate from the non-negativity of ranks in Definition~\ref{maindef}.
It remains to consider the two cases \(\bullet=\tau\delta\) and \(\bullet=\tau\delta\tau\).

\medskip\noindent
\textbf{Case \(\bullet=\tau\delta\):}
For any \(v\in\ker M[A]\), we have \[(M+I_A)[A]\, v = (M[A] + I_{A})v = v.\] Then \(\ker M[A]\subseteq\operatorname{col}(M[A]+I_{A})\). Hence
\[
\begin{aligned}
\operatorname{corank}(M[A])
&= \dim\ker M[A] \le \dim\operatorname{col}(M[A]+I_{A})= \operatorname{rank}(M[A]+I_{A})\\
&=\operatorname{rank}((M+I_A)[A]) \le \operatorname{rank}(M+I_A).
\end{aligned}
\]
Therefore,
\[
r_{\langle\tau\delta\rangle}(M,A)=\operatorname{rank}(M+I_A)-\operatorname{corank}(M[A])\ge 0.
\]

\medskip\noindent
\textbf{Case \(\bullet=\tau\delta\tau\):}
For any \(v\in\ker (M[A]+I_{A})\), we have $M[A]v=-v.$ Thus \(\ker (M[A]+I_{A})\subseteq\operatorname{col}(M[A])\). Hence
\[
\begin{aligned}
\operatorname{rank}(M)&\ge\operatorname{rank}(M[A])=\dim \operatorname{col} M[A]\ge\ker (M[A]+I_{A})\\
&=\operatorname{corank}((M[A]+I_{A}))=\operatorname{corank}((M+I_A)[A]).
\end{aligned}
\]
Therefore,  \[r_{\langle\tau\delta\tau\rangle}(M,A)=\operatorname{rank}(M)-\operatorname{corank}((M+I_A)[A])\ge 0.\]
\end{proof}

\begin{lemma}[Combinatorial Nullstellensatz \cite{Alon}]\label{Alontheorem}
Let $\mathbb{F}$ be a field, let $d_1,\dots,d_n$ be non-negative integers, and let $P \in \mathbb{F}[x_1,\dots,x_n]$ be a polynomial. Suppose that $\deg P = \sum_{i=1}^n d_i$, and  the coefficient of $x_1^{d_1}\dots x_n^{d_n}$ is nonzero.
Then for every subsets $S_1,\dots,S_n$ of $\mathbb{F}$ with $|S_i| > d_i$, there exists $(s_1,\dots,s_n) \in S_1 \times \dots \times S_n$ such that
\[
P(s_1,\dots,s_n) \neq 0.
\]
\end{lemma}

\begin{lemma}\label{fullrank}
Let $M\in \mathbb{F}^{n\times n}$ be a matrix.
Then there exists a diagonal matrix $D=\operatorname{diag}(c_1,\dots,c_n)$ with $c_i\in\{0,1\}$ for $i=1,\dots,n$ such that $M+D$ is non-singular.
\end{lemma}

\begin{proof}
Consider the determinant of $M+D$ as a polynomial in the indeterminates $c_1,\dots,c_n$ with coefficients in $\mathbb{F}$:
\[
f(c_1,\dots,c_n)=\det\!\bigl(M+\operatorname{diag}(c_1,\dots,c_n)\bigr)\in\mathbb{F}[c_1,\dots,c_n].
\]
Expand $f$ via the Leibniz formula:
\[
f(c_1,\dots,c_n)=\sum_{\pi\in S_n}\operatorname{sgn}(\pi)\prod_{i=1}^n\bigl(M+\operatorname{diag}(c_1,\dots,c_n)\bigr)_{i,\pi(i)}.
\]
For the identity permutation $\pi=\mathrm{id}$, the product equals
\[
\prod_{i=1}^n (M_{ii}+c_i).
\]
In this product the monomial $c_1c_2\cdots c_n$ appears with coefficient $1$. Any non-identity permutation $\pi$ satisfies $\pi(i)\neq i$ for at least one index $i$; for such an $i$ the factor $(M+\operatorname{diag}(c_1,\dots,c_n))_{i,\pi(i)}$ is the off-diagonal entry $M_{i,\pi(i)}$, which is independent of all $c_j$. Consequently the monomial $c_1\cdots c_n$ does not occur in the contribution of any $\pi\neq\mathrm{id}$. Hence $f$ contains the monomial $c_1\cdots c_n$ with coefficient $1$ and is therefore a nonzero polynomial.

We now apply Lemma~\ref{Alontheorem} with $d_i=1$ for all $i$. The polynomial $f$ has total degree $n$ and the coefficient of $c_1\cdots c_n$ is nonzero. For each $i$ take $S_i = \{0,1\}\subseteq\mathbb{F}$, which satisfies $|S_i| = 2 > 1 = d_i$. By Lemma \ref{Alontheorem}, there exists a tuple $(\bar c_1,\dots,\bar c_n)\in\{0,1\}^n$ such that $f(\bar c_1,\dots,\bar c_n)\neq 0$, i.e.,
\[
\det\!\bigl(M+\operatorname{diag}(\bar c_1,\dots,\bar c_n)\bigr)\neq 0.
\]
Thus $M+\operatorname{diag}(\bar c_1,\dots,\bar c_n)$ is non-singular. Setting $c_i = \bar c_i$ for $i=1,\dots,n$ yields the desired diagonal matrix $D$.
\end{proof}

\begin{proposition}\label{new:lem05}
Let \(M\) be a square matrix over a field, with rows and columns indexed by a finite set \(V\). Then
\begin{enumerate}
    \item[$(1)$] The partial-$\langle \tau \rangle$ polynomial satisfies
        \[
        \deg P_{\langle \tau \rangle}(M, z) = |V|.
        \]
    \item[$(2)$] The partial-$\langle \delta \rangle$, $\langle \delta\tau \rangle$, $\langle \tau\delta \rangle$ and $\langle \tau\delta\tau \rangle$ polynomials satisfy
        \[
        \deg P_{\langle \bullet \rangle}(M, z) \leq |V|, \quad \text{for } \bullet \in \{\delta, \delta\tau, \tau\delta, \tau\delta\tau\}.
        \]
\end{enumerate}

\end{proposition}

\begin{proof}
We first establish the upper bounds for all five polynomials. For any subset $A \subseteq V$,
\begin{align*}
&\operatorname{rank}(M[A]) + \operatorname{rank}(M[A^c]) \leq |A| + |A^c| = |V|, \\
&\operatorname{rank}(M+ I_A) \leq |V|, \\
&\operatorname{rank}((M+ I_A)[A]) + \operatorname{rank}(M[A^c]) \leq |A| + |A^c| = |V|, \\
&\operatorname{rank}(M+ I_A) - \operatorname{corank}(M[A]) \leq \operatorname{rank}(M + I_A) \leq |V|, \\
&\operatorname{rank}(M)- \operatorname{corank}((M+ I_A)[A]) \leq \operatorname{rank}(M) \leq |V|.
\end{align*}
Consequently,
\[
\deg P_{\langle \bullet \rangle}(M, z)  \leq |V| \quad \text{for } \bullet \in \{\delta, \tau, \delta\tau, \tau\delta, \tau\delta\tau\}.
\]
In particular, statement (2) follows directly. It remains to show that the bound $|V|$ is attained for $\langle \tau \rangle$.

By Lemma~\ref{fullrank}, there exists a diagonal matrix $D = \operatorname{diag}(d_1,\dots,d_{|V|})$ with $d_i\in\{0,1\}$ such that $M + D$ is non-singular. Set $A = \{ i \mid d_i = 1 \}$. Then $\operatorname{rank}(M + I_A) = |V|$, which gives the term $z^{|V|}$ in $P_{\langle \tau \rangle}(M, z)$.  Thus $\deg P_{\langle \tau \rangle}(M, z) = |V|$.
\end{proof}

\section{Interpolation behavior}

Given a polynomial \(P(z) = \sum_i c_i z^i \), let its support be \[\operatorname{supp}(P) = \{ i \mid c_i \neq 0 \}. \] We say that \( P(z) \) has a \emph{gap of size \( k \)} if there exists an integer \( i \) such that \( i-1, i+k \in \operatorname{supp}(P) \) but \( i, i+1, \dots, i+k-1 \notin \operatorname{supp}(P) \). In particular, when \( k \ge 2 \) we say that \( P(z) \) has a gap of size at least \(2\). If \( p(z) \) is nonzero and has no gaps, it is called \emph{interpolating}. Expressing \( P(z) \) as the sum of its even-degree and odd-degree parts, \( P(z) = P_e(z^2) + z\, P_o(z^2) \), we call \( P(z) \) \emph{even-interpolating} (resp.\ \emph{odd-interpolating}) when \( P_e \) (resp.\ \( P_o \)) is nonzero and interpolating. Finally, \( P(z) \) is an \emph{even polynomial} (resp.\ \emph{odd polynomial}) if its support contains only even (resp.\ odd) numbers, i.e., \( P_o=0 \) (resp.\ \( P_e=0 \)).

\begin{lemma}\label{the01}
Let \(M\) be a square matrix over a field $\mathbb{F}$, with rows and columns indexed by a finite set \(V\), let $v\in V$ and let $S_0\subseteq V$ with $v\notin S_0$.
Set $S_1 = S_0 \cup \{v\}$. Then the following hold:

\begin{enumerate}
\item[$(1)$] $\displaystyle \operatorname{rank}((M+I_{S_1})[S_1]) -
          \operatorname{rank}((M+I_{S_0})[S_0]) \in \{0, 1, 2\}$.

\item[$(2)$] $\displaystyle \operatorname{rank}(M[S_1]) - \operatorname{rank}(M[S_0]) \in \{0,1,2\}$.
\end{enumerate}
\end{lemma}

\begin{proof} {\bf (1)}
Let $M_0 = (M+I_{S_0})[S_0]$ and $M_1 = (M+I_{S_1})[S_1]$. Since $S_1 = S_0 \cup \{v\}$ with $v \notin S_0$, we can write $M_1$ in block form
\[
M_1 = \begin{bmatrix} a & \alpha^{\!\top} \\ \beta & M_0 \end{bmatrix},
\]
where $a = M[v,v]+1$, and $\alpha = (\alpha_1,\dots,\alpha_{|S_0|})^{\!\top}$ and $\beta = (\beta_1,\dots,\beta_{|S_0|})^{\!\top}$ are column vectors of length $|S_0|$ over $\mathbb{F}$. Consider the row vectors of $M_1$:
\[
\mathbf r_0 = (a,\; \alpha^{\!\top}), \qquad \mathbf r_i = (\beta_i,\; \mathbf m_i) \;\; (i=1,\dots,|S_0|),
\]
where $\mathbf m_i$ is the $i$-th row of $M_0$. Let $V$ be the subspace spanned by $$\{\mathbf r_1,\dots,\mathbf r_{|S_0|}\}.$$

Observe that $V$ is the row space of the matrix $[\beta \ M_0]$, which is obtained by adding the column $\beta$ to the left of $M_0$. Since adding a new column increases the rank by at most one, we have
\[
\dim V = \operatorname{rank}([\beta \ M_0]) \in \{\operatorname{rank}(M_0),\; \operatorname{rank}(M_0)+1\}.
\]

Adding $\mathbf r_0$ to $V$ can increase the dimension by at most one, so
\[
\operatorname{rank}(M_1) = \dim \operatorname{span}(V \cup \{\mathbf r_0\}) \le \dim V + 1,
\]
and therefore $\operatorname{rank}(M_1) \in \{\operatorname{rank}(M_0),\; \operatorname{rank}(M_0)+1,\; \operatorname{rank}(M_0)+2\}$.
This completes the proof of (1).

\noindent{\bf (2)}
Set $N_0 = M[S_0]$ and $N_1 = M[S_1]$.
Since $S_1 = S_0 \cup \{v\}$, we can write $N_1$ in block form
\[
N_1 = \begin{bmatrix} b & \gamma^{\!\top} \\ \delta & N_0 \end{bmatrix},
\]
where $b = M[v,v]$, $\gamma^{\!\top}$ is the row vector $M[v, S_0]$, and $\delta$ is the column vector $M[S_0, v]$. The same row-space analysis used in the proof of (1) applied to $N_0$ and $N_1$ shows that
\[
\operatorname{rank}(N_1) \in \{\operatorname{rank}(N_0),\; \operatorname{rank}(N_0)+1,\; \operatorname{rank}(N_0)+2\}.
\]
Hence $\operatorname{rank}(N_1) - \operatorname{rank}(N_0) \in \{0,1,2\}$, which proves (2).
\end{proof}

\begin{theorem}\label{twointerpolating}
	Let \(M\) be a square matrix over a field, with rows and columns indexed by a finite set \(V\). Then
\begin{enumerate}
\item[$(1)$] $P_{\langle \bullet \rangle}(M,z)$ is interpolating for $\bullet \in \left\{\tau, \tau\delta\tau\right\}$.
\item[$(2)$] For $\bullet \in \left\{\delta,\delta\tau,\tau\delta\right\}$, $P_{\langle \bullet \rangle}(M,z)$ has no gaps of size $2$ or larger. In particular, if $P_{\langle \bullet \rangle}(M,z)$ is an even (resp. odd) polynomial, then it is even-interpolating (resp. odd-interpolating).
\item[$(3)$] If \(M\) is symmetric and has zero diagonal, then \(P_{\langle \delta \rangle}(M,z)\) is even-interpolating.
\end{enumerate}
\end{theorem}

\begin{proof}

Take any $A_0\subseteq V$ and $v\in V\setminus A_0$, and set $A_1=A_0\cup\{v\}$. We estimate the absolute value of $r_{\langle \bullet \rangle}(M, A_1)-r_{\langle \bullet \rangle}(M, A_0)$.

\medskip\noindent
\textbf{Case $\bullet=\tau$:}
We have $$r_{\langle \tau \rangle}(M, A_1)-r_{\langle \tau \rangle}(M, A_0)=\operatorname{rank}(M+I_{A_1})-\operatorname{rank}(M+I_{A_0}).$$
Set $\mathbf A_0=M+I_{A_0}$ and $\mathbf A_1=M+I_{A_1}$.
Since $A_1=A_0\cup{v}$, the set $A_1$ differs from $A_0$ exactly at $v$, so $\mathbf A_1$ and $\mathbf A_0$ coincide everywhere except at the diagonal entry corresponding to $v$. Thus $\mathbf A_1 = \mathbf A_0 + \mathbf u_v\mathbf u_v^{\!\top}$, where $\mathbf u_v$ is the standard basis vector with a $1$ in the $v$-th coordinate and $0$ elsewhere.
Let $C_0$ and $C_1$ be the column spaces of $\mathbf A_0$ and $\mathbf A_1$, respectively.
Because each column of $\mathbf A_1$ is either equal to the corresponding column of $\mathbf A_0$ or differs by $\mathbf u_v$, we have $C_1 \subseteq C_0 + \operatorname{span}{\mathbf u_v}$, whence $\dim C_1 \le \dim C_0 + 1$.
Symmetrically, $\mathbf A_0 = \mathbf A_1 - \mathbf u_v\mathbf u_v^{\!\top}$ gives $C_0 \subseteq C_1 + \operatorname{span}{\mathbf u_v}$, so $\dim C_0 \le \dim C_1 + 1$.
Thus $|\dim C_1 - \dim C_0| \le 1$, i.e., $|\operatorname{rank}(\mathbf A_1) - \operatorname{rank}(\mathbf A_0)| \le 1$.
Hence $|r_{\langle \tau \rangle}(M, A_1)-r_{\langle \tau \rangle}(M, A_0)|\le 1$.

\medskip\noindent
\textbf{Case $\bullet=\tau\delta\tau$:}
We have $r_{\langle \tau\delta\tau\rangle}(M, A)=\operatorname{rank}(M)-\operatorname{corank}((M+I_A)[A])$.
Note that
\[
\begin{aligned}
&~~~~\operatorname{corank}((M+I_{A_1})[A_1])-\operatorname{corank}((M+I_{A_0})[A_0]) \\
&= \bigl(|A_1|-\operatorname{rank}((M+I_{A_1})[A_1])\bigr) - \bigl(|A_0|-\operatorname{rank}((M+I_{A_0})[A_0])\bigr) \\
&= 1 - \bigl(\operatorname{rank}((M+I_{A_1})[A_1])-\operatorname{rank}((M+I_{A_0})[A_0])\bigr).
\end{aligned}
\]
By Lemma \ref{the01} (1), we have
\[
\operatorname{rank}((M+I_{A_1})[A_1])-\operatorname{rank}((M+I_{A_0})[A_0])\in\{0,1,2\}.
\]
Consequently,
\[
\operatorname{corank}((M+I_{A_1})[A_1])-\operatorname{corank}((M+I_{A_0})[A_0])\in\{1,0,-1\},
\]
so
\[
|\operatorname{corank}((M+I_{A_1})[A_1])-\operatorname{corank}((M+I_{A_0})[A_0])|\le 1.
\]
Since $\operatorname{rank}(M)$ is constant, we obtain $|r_{\langle \tau\delta\tau\rangle}(M_1, A_1)-r_{\langle \tau\delta\tau\rangle}(M_0, A_0)|\le 1$.

\medskip\noindent
\textbf{Case $\bullet=\delta$:}
We have
\[
\begin{aligned}
r_{\langle \delta \rangle}(M, A_1) - r_{\langle \delta \rangle}(M, A_0)
&= \bigl(\operatorname{rank}(M[A_1]) + \operatorname{rank}(M[A_1^c])\bigr) \\
&\quad - \bigl(\operatorname{rank}(M[A_0]) + \operatorname{rank}(M[A_0^c])\bigr).
\end{aligned}
\]
By Lemma \ref{the01} (2), we have
\[
\operatorname{rank}(M[A_1]) - \operatorname{rank}(M[A_0]) \in \{0,1,2\}.
\]
Notice that $A_0^c = A_1^c \cup \{v\}$. Applying Lemma \ref{the01} (2) again gives
\[
\operatorname{rank}(M[A_0^c]) - \operatorname{rank}(M[A_1^c]) \in \{0,1,2\},
\]
i.e.,
\[
\operatorname{rank}(M[A_1^c]) - \operatorname{rank}(M[A_0^c]) \in \{0,-1,-2\}.
\]
Adding the two differences we obtain
\[
r_{\langle \delta \rangle}(M, A_1) - r_{\langle \delta \rangle}(M, A_0) \in \{0,\pm1,\pm2\},
\]
so
\[
|r_{\langle \delta \rangle}(M, A_1) - r_{\langle \delta \rangle}(M, A_0)| \le 2.
\]

\medskip\noindent
\textbf{Case $\bullet=\delta\tau$:}
We have
\[
\begin{aligned}
r_{\langle \delta\tau \rangle}(M, A_1) - r_{\langle \delta\tau \rangle}(M, A_0)
&= \bigl(\operatorname{rank}((M+I_{A_1})[A_1]) + \operatorname{rank}(M[A_1^c])\bigr) \\
&~- \bigl(\operatorname{rank}((M+I_{A_0})[A_0]) + \operatorname{rank}(M[A_0^c])\bigr).
\end{aligned}
\]
By Lemma \ref{the01} (1), we have
\[
\operatorname{rank}((M+I_{A_1})[A_1]) - \operatorname{rank}((M+I_{A_0})[A_0]) \in \{0,1,2\}.
\]
As in the previous case,
\[
\operatorname{rank}(M[A_1^c]) - \operatorname{rank}(M[A_0^c]) \in \{0,-1,-2\}.
\]
Therefore
\[
r_{\langle \delta\tau \rangle}(M, A_1) - r_{\langle \delta\tau \rangle}(M, A_0) \in \{0,\pm1,\pm2\},
\]
and thus
\[
|r_{\langle \delta\tau \rangle}(M, A_1) - r_{\langle \delta\tau \rangle}(M, A_0)| \le 2.
\]

\medskip\noindent
\textbf{Case $\bullet=\tau\delta$:}
We have
\[
\begin{aligned}
r_{\langle \tau\delta \rangle}(M, A_1) - r_{\langle \tau\delta \rangle}(M, A_0)
&= \bigl(\operatorname{rank}(M + I_{A_1}) - \operatorname{rank}(M + I_{A_0})\bigr) \\
&~- \bigl(\operatorname{corank}(M[A_1]) - \operatorname{corank}(M[A_0])\bigr).
\end{aligned}
\]
From the case $\bullet=\tau$ we know $a:=\operatorname{rank}(M + I_{A_1}) - \operatorname{rank}(M + I_{A_0}) \in \{0, \pm 1\}$.
By Lemma \ref{the01} (2), $b:=\operatorname{rank}(M[A_1]) - \operatorname{rank}(M[A_0]) \in \{0, 1, 2\}$.
Now
\[
\begin{aligned}
&~~~~~\operatorname{corank}(M[A_1]) - \operatorname{corank}(M[A_0])\\
&= \bigl(|A_1| - \operatorname{rank}(M[A_1])\bigr) - \bigl(|A_0| - \operatorname{rank}(M[A_0])\bigr) \\
&= 1 - b.
\end{aligned}
\]
Thus $r_{\langle \tau\delta \rangle}(M, A_1) - r_{\langle \tau\delta \rangle}(M, A_0) = a - (1 - b) = a + b - 1$.
Since $a \in \{-1,0,1\}$ and $b \in \{0,1,2\}$, we have $a + b \in \{-1,0,1,2,3\}$, so $a + b - 1 \in \{-2, -1, 0, 1, 2\}$. Hence, $|r_{\langle \tau\delta \rangle}(M, A_1) - r_{\langle \tau\delta \rangle}(M, A_0)| \le 2$.

\medskip\noindent
{\bf (1) and (2)}
Let $V$ be the index set of the matrix $M$.
Consider the graph $H$ whose vertices are all subsets of $V$, with an edge between two subsets $X$ and $Y$ if $|X \Delta  Y| = 1$. Clearly,
 $H$ is connected.

For $\bullet\in\{\tau,\tau\delta\tau\}$, for any two subsets $X,Y\subseteq V$ with $|X\Delta Y|=1$, the arguments given in
the two cases above show that
$$|r_{\langle\bullet\rangle}(M,X)-r_{\langle\bullet\rangle}(M,Y)|\le 1.$$
Let $\mathcal{E}_\bullet=\{r_{\langle\bullet\rangle}(M,A) \mid A\subseteq V\}$,
and set $m=\min\mathcal{E}_\bullet$ and $M=\max\mathcal{E}_\bullet$.
Because $H$ is connected, there exists a path
$S_0,S_1,\dots,S_t$ in $H$ with $r_{\langle\bullet\rangle}(M,S_0)=m$ and
$r_{\langle\bullet\rangle}(M,S_t)=M$.
Since $|r_{\langle\bullet\rangle}(M,S_{i+1})-r_{\langle\bullet\rangle}(M,S_i)|\le 1$
for each $i$, the sequence $\{r_{\langle\bullet\rangle}(M,S_i)\}_{i=0}^{t}$ must take
every integer value between $m$ and $M$.
Hence $\mathcal{E}_\bullet = \{m,m+1,\dots,M\}$, which means that $P_{\langle\bullet\rangle}(M,z)$ is interpolating.

For $\bullet\in\{\delta,\delta\tau,\tau\delta\}$,
for any two subsets $X,Y\subseteq V$ with $|X\Delta Y|=1$, the three cases treated above show that $$|r_{\langle\bullet\rangle}(M,X)-r_{\langle\bullet\rangle}(M,Y)|\le 2.$$
 Now suppose that $\mathcal{E}_\bullet$ has a gap of size at least $2$.
Thus there exist integers $k$ and $k+d$ with $d\ge 3$ such that $k,k+d\in\mathcal{E}_\bullet$, but $k+1,k+2\notin\mathcal{E}_\bullet$.
Choose $A,B\subseteq V$ with $r_{\langle\bullet\rangle}(M,A)=k$ and $r_{\langle\bullet\rangle}(M,B)=k+d$.
Because $H$ is connected, there is a path $S_0=A,S_1,\dots ,S_t=B$ in $H$.
Let $r_i = r_{\langle\bullet\rangle}(M,S_i)$ for $i=0,\dots,t$.
We have $r_0=k$, $r_t=k+d$ with $d\ge 3$, and $|r_{i+1}-r_i|\le 2$ for all $i$.
Consider the smallest index $j$ such that $r_j > k$ (such $j$ exists because $r_t > k$).
Then $r_{j-1}=k$ and $1 \le r_j - k \le 2$, so $r_j = k+1$ or $k+2$.
But $k+1$ and $k+2$ are not in $\mathcal{E}_\bullet$, contradicting that $r_j\in\mathcal{E}_\bullet$.
Therefore $\mathcal{E}_\bullet$ cannot contain a gap of size $2$ or larger. Hence, $P_{\langle\bullet\rangle}(M,z)$ has no gaps of size $2$ or larger.

In particular, if $P_{\langle\bullet\rangle}(M,z)$ is an even (resp. odd) polynomial, then all exponents in $\mathcal{E}_\bullet$ are even (resp. odd). Since two distinct even (resp. odd) numbers differ by at least $2$, and there are no gaps of size $2$ or larger, the exponents must form a set of consecutive even (resp. odd) integers. Thus $P_{\langle\bullet\rangle}(M,z)$ is even-interpolating (resp. odd-interpolating).

\medskip\noindent{\bf (3)}
Assume that \(M\) is symmetric and has zero diagonal. For any subset \(A\subseteq V\), the principal submatrices \(M[A]\) and \(M[A^c]\) are also symmetric with zero diagonal. A symmetric matrix with zero diagonal has even rank. Therefore \(\operatorname{rank}(M[A])\) and \(\operatorname{rank}(M[A^c])\) are even. Consequently,
\(r_{\langle\delta\rangle}(M,A)=\operatorname{rank}(M[A])+\operatorname{rank}(M[A^c])\) is even for every \(A\). Thus \(P_{\langle\delta\rangle}(M,z)\) is an even polynomial. By (2), an even polynomial for \(\bullet=\delta\) is even-interpolating. Hence \(P_{\langle\delta\rangle}(M,z)\) is even-interpolating.
\end{proof}

\section{Pivot invariance and inverse duality}

Let $M$ be a square matrix over a field with rows and columns indexed by a finite set $V$, and let \( X \subseteq V \) be such that \(M[X] \) is non-singular. The \emph{pivot} of \( M \) on \( X \), denoted by \( M * X \), is defined as follows. Let
\[
M = \begin{pmatrix} P & Q \\ R & S \end{pmatrix}
\]
with \( P = M[X] \). Then
\[
M*X = \begin{pmatrix} P^{-1} & -P^{-1}Q \\ RP^{-1} & S - RP^{-1}Q \end{pmatrix}.
\]

\begin{lemma}[\cite{Brijder}]\label{pivotrankinvariant}
Let $M$ be a square matrix over a field with rows and columns indexed by a finite set $V$, and let \( X \subseteq V \) be such that \( M[X] \) is non-singular. Then, for all \( A \subseteq V \),
$$
\operatorname{corank} M[A] \;=\; \operatorname{corank}\big((M*X)[A\Delta X]\big).
$$
\end{lemma}

\begin{theorem}\label{thm:pivot-invariance}
Let $M$ be a square matrix over a field with rows and columns indexed by a finite set $V$, and let $X \subseteq V$ be such that  $M[X]$ is non-singular. Then
\[
P_{\langle \delta \rangle}(M, z) = P_{\langle \delta \rangle}(M \ast X, z).
\]
In particular, if $M$ is invertible, then
\[
P_{\langle \delta \rangle}(M, z) = P_{\langle \delta \rangle}(M \ast V, z) = P_{\langle \delta \rangle}(M^{-1}, z).
\]
\end{theorem}

\begin{proof}
For any $A \subseteq V$, let $B = A \,\Delta\, X$. Then $B^c=A^c\,\Delta\, X$. By Lemma~\ref{pivotrankinvariant},
\[
\operatorname{corank} M[A] = \operatorname{corank}\bigl((M \ast X)[B]\bigr),\]
and
\[\operatorname{corank} M[A^{c}] = \operatorname{corank}\bigl((M \ast X)[B^{c}]\bigr).
\]
Then
\[
\begin{aligned}
&~\operatorname{rank}(M[A])+\operatorname{rank}(M[A^{c}])\\
&= |V|-\operatorname{corank}(M[A])-\operatorname{corank}(M[A^{c}])\\
&= |V|-\operatorname{corank}\bigl((M\ast X)[B]\bigr)-\operatorname{corank}\bigl((M\ast X)[B^{c}]\bigr)\\
&= \operatorname{rank}\bigl((M\ast X)[B]\bigr)+\operatorname{rank}\bigl((M\ast X)[B^{c}]\bigr).
\end{aligned}
\]
Therefore,
\[
\begin{aligned}
P_{\langle \delta \rangle}(M, z)
&= \sum_{A \subseteq V} z^{\operatorname{rank}(M[A])+\operatorname{rank}(M[A^{c}])} \\
&= \sum_{B \subseteq V} z^{\operatorname{rank}((M\ast X)[B])+\operatorname{rank}((M\ast X)[B^{c}])} \\
&= P_{\langle \delta \rangle}(M \ast X, z).
\end{aligned}
\]

The invertible case follows by taking $X = V$ and noting that $M \ast V = M^{-1}$.
\end{proof}

Let $\mathbf{A}$ be the symmetric adjacency matrix over $\mathrm{GF}(2)$ of a graft $(G, L_G)$. For any $X \subseteq V(G)$ such that $\mathbf{A}[X]$ is invertible, the pivot $\mathbf{A} \ast X$ is defined. Note that in $\mathrm{GF}(2)$ the pivot preserves symmetry: $(\mathbf{A} \ast X)^{\!\top} = \mathbf{A} \ast X$ \cite{BRHH}. Hence, $\mathbf{A} \ast X$ is the adjacency matrix of another graft.
Combining this fact with Theorem \ref{thm:pivot-invariance}, we therefore obtain a corollary for grafts.

\begin{corollary}\label{cor:graft-pivot-invariance}
     Let $(G,L_G)$ be a graft with adjacency matrix $\mathbf{A}$. Let $X\subseteq V(G)$. If $\mathbf{A}*X$ is defined, denote by $(G',L_{G'})$ the graft with adjacency matrix $\mathbf{A}\ast X$. Then
\[P_{(G,L_G)}^{\langle \delta \rangle}(z)=P_{(G',L_{G'})}^{\langle \delta \rangle}(z).\]
\end{corollary}

\begin{remark}
Pivoting on a matrix corresponds to partial duality on its associated bouquet. Consequently, for bouquets $B_1$ and $B_2$ whose intersection grafts are related by a matrix pivot, Theorem~\ref{thm:main-equivalence} together with Corollary~\ref{cor:graft-pivot-invariance} implies that if $B_1$ and $B_2$ are partial duals of each other (i.e., $B_2 = B_1^{\delta(A)}$ for some $A\subseteq E(B_1)$), then ${}^{\partial}\varepsilon_{B_1}^{\delta}(z) = {}^{\partial}\varepsilon_{B_2}^{\delta}(z)$.
\end{remark}

Gross, Mansour and Tucker \cite{GMT2} showed that partial-$\delta\bullet\delta$ and partial-$\bullet$ polynomials are closely related:

\begin{theorem}[\cite{GMT2}]\label{thm:gmt-duality}
    Let $G$ be a ribbon graph. Then
    $${}^{\partial}\varepsilon_{G}^{\delta\bullet\delta}(z) = {}^{\partial}\varepsilon_{G^*}^{\bullet}(z).$$In particular, ${}^{\partial}\varepsilon_{G}^{\tau\delta\tau}(z) = {}^{\partial}\varepsilon_{G^*}^{\tau}(z)$ and ${}^{\partial}\varepsilon_{G}^{\delta\tau}(z) = {}^{\partial}\varepsilon_{G^*}^{\tau\delta}(z).$
\end{theorem}

When passing to the matrix setting, recall that for a non-singular matrix $M$, the inverse $M^{-1}$ plays the role of the geometric dual (since pivoting on the whole vertex set gives $M^{-1}$).
Theorem~\ref{thm:pivot-invariance} already shows that the partial-$\langle\delta\rangle$ polynomial is invariant under inversion:
\[
P_{\langle\delta\rangle}(M,z)=P_{\langle\delta\rangle}(M^{-1},z).
\]
For the other partial-$\langle\bullet\rangle$ polynomials, the following theorem establishes duality relations between $M$ and $M^{-1}$ that are analogous to the ribbon graph identities in Theorem~\ref{thm:gmt-duality}.

\begin{theorem}\label{thm:matrix-duality}
Let \(M\) be a non-singular square matrix over a field, with rows and columns indexed by a finite set \(V\). Then
\begin{enumerate}
\item[$(1)$] $P_{\langle\tau\delta\tau\rangle}(M, z) = P_{\langle\tau\rangle}(M^{-1}, z),$

\item[$(2)$] $P_{\langle\delta\tau\rangle}(M, z) = P_{\langle\tau\delta\rangle}(M^{-1}, z).$
\end{enumerate}
\end{theorem}

\begin{proof} {\bf (1)}
Let $n=|V|$. For any $A\subseteq V$ write $M$ in block form with respect to the partition $(A^c,A)$:
\[
M=\begin{bmatrix}P&Q\\[2pt]R&S\end{bmatrix},
\]
where $P$ is the submatrix on $A^c$, $S$ on $A$, $Q$ is the $A^c\times A$ block, and $R$ is the $A\times A^c$ block.
Then $(M+I_A)[A]=S+I$.  Because $M$ is non-singular, $\operatorname{rank}(M)=n$, and
\[
\operatorname{rank}(M)-\operatorname{corank}((M+I_A)[A])
 = n-\bigl(|A|-\operatorname{rank}(S+I)\bigr)
 = |A^c|+\operatorname{rank}(S+I).
\]

On the other hand, the partial-$\langle\tau\rangle$ polynomial of $M^{-1}$ contains the term
$\operatorname{rank}(M^{-1}+I_A)$. Left-multiplying by the non-singular matrix $M$ preserves rank, so
\[
\operatorname{rank}(M^{-1}+I_A)=\operatorname{rank}(I+M I_A).
\]
With $I_A=\begin{bmatrix}0&0\\[2pt]0&I\end{bmatrix}$ we have
\[
M I_A=\begin{bmatrix}0&Q\\[2pt]0&S\end{bmatrix},\qquad
I+M I_A=\begin{bmatrix}I_{A^c}&Q\\[2pt]0&S+I\end{bmatrix}.
\]
Hence
\[
\operatorname{rank}(I+M I_A)=\operatorname{rank}(I_{A^c})+\operatorname{rank}(S+I)
 =|A^c|+\operatorname{rank}(S+I).
\]
Thus for every $A\subseteq V$,
\[
\operatorname{rank}(M)-\operatorname{corank}((M+I_A)[A])=\operatorname{rank}(M^{-1}+I_A).
\]
By Definition \ref{maindef}, $r_{\langle\tau\delta\tau\rangle}(M,A)=r_{\langle\tau\rangle}(M^{-1},A)$. Summing over all $A\subseteq V$ gives
\[
P_{\langle\tau\delta\tau\rangle}(M,z)=P_{\langle\tau\rangle}(M^{-1},z).
\]

\noindent{\bf (2)} We need to show that for every \(A\subseteq V\),
\[
r_{\langle\delta\tau\rangle}(M,A)=r_{\langle\tau\delta\rangle}(M^{-1},A).
\]
From Definition \ref{maindef}, we have
\[
r_{\langle\delta\tau\rangle}(M,A)=\operatorname{rank}((M+I_A)[A])+\operatorname{rank}(M[A^c])
\]
and
\[
r_{\langle\tau\delta\rangle}(M^{-1},A)=\operatorname{rank}(M^{-1}+I_A)-\operatorname{corank}(M^{-1}[A]),
\]
where $M^{-1}[A]$ denotes the principal submatrix of the inverse matrix $M^{-1}$ on the index set $A$.
The proof of part (1) established
\[
\operatorname{rank}(M^{-1}+I_A)=\operatorname{rank}((M+I_A)[A])+|A^c|,
\]
hence
\[
\begin{aligned}
r_{\langle\tau\delta\rangle}(M^{-1},A)
=\operatorname{rank}((M+I_A)[A])+|A^c|-\operatorname{corank}(M^{-1}[A]).
\end{aligned}
\]
Thus \(r_{\langle\delta\tau\rangle}(M,A)=r_{\langle\tau\delta\rangle}(M^{-1},A)\) is equivalent to
\[\operatorname{rank}(M[A^c])=|A^c|-\operatorname{corank}(M^{-1}[A]),\]
that is, \[\operatorname{corank}(M[A^c])=\operatorname{corank}(M^{-1}[A]),\]
which comes from Lemma \ref{pivotrankinvariant} by taking $X=V$.
\end{proof}

\begin{remark}
Theorem~\ref{thm:matrix-duality} does not extend to matrix pivots on arbitrary proper subsets. That is, for a proper subset $X \subsetneq V$ with $M[X]$ non-singular, the identities
\[
P_{\langle\tau\delta\tau\rangle}(M, z) = P_{\langle\tau\rangle}(M*X, z) \quad \text{and} \quad
P_{\langle \delta\tau\rangle}(M, z) = P_{\langle\tau\delta\rangle}(M*X, z)
\]
do not hold in general. For example, let \(M\) be the adjacency matrix of the wheel graph \(W_5\) on \(6\) vertices $\{0, 1, 2, 3, 4, 5\}$,
whose edge set is
\[
E(W_5)=\bigl\{(0, 1), (0, 2), (0, 3), (0, 4), (0, 5), (1, 2), (2, 3), (3, 4), (4, 5), (5, 1)\bigr\}.
\]
The pivot \(M*\{1,2\}\) is the adjacency matrix of a graph on the same vertex set with edges
\[
E=\bigl\{(0, 1), (0, 2), (0, 4), (1, 2), (1, 3), (3, 4), (4, 5), (2, 5), (5, 3)\bigr\}.
\]
Direct computation gives
\begin{align*}
P_{\langle\tau\delta\tau\rangle}(M, z)    &= z^{4}\bigl(33z^{2} + 26z + 5\bigr),\\[1mm]
  P_{\langle\tau\rangle}(M*\{1,2\}, z)      &= z^{4}\bigl(35z^{2} + 25z + 4\bigr),\\[1mm]
  P_{\langle\delta\tau\rangle}(M, z)        &= z^{3}\bigl(11z^{3} + 33z^{2} + 5z + 15\bigr),\\[1mm]
  P_{\langle\tau\delta\rangle}(M*\{1,2\}, z)&= z^{3}\bigl(12z^{3} + 20z^{2} + 31z + 1\bigr).
\end{align*}

\end{remark}

\begin{corollary}
     Let $(G, L_G)$ be a graft with non-singular adjacency matrix $\mathbf{A}$, and let $(G', L_{G'})$ be the graft corresponding to $\mathbf{A}^{-1}$. Then
\[P_{(G,L_G)}^{\langle \tau\delta\tau \rangle}(z)=P_{(G',L_{G'})}^{\langle \tau \rangle}(z)\quad \text{and} \quad P_{(G,L_G)}^{\langle \delta\tau \rangle}(z)=P_{(G',L_{G'})}^{\langle \tau\delta \rangle}(z).\]
\end{corollary}

\begin{remark}
The partial-$\langle\tau\delta\tau\rangle$ polynomial is closely related to the interlace polynomial of a square matrix~\cite{Brijder}.
For a square matrix $M$ over a field, with rows and columns indexed by a finite set $V$, the interlace polynomial $q(M,x)$ of $M$ is defined as
\[
q(M,x)=\sum_{A \subseteq V} (x-1)^{\operatorname{corank}(M[A])}.
\]
One can verify that
\[
P_{\langle \tau\delta\tau \rangle}(M,z)=z^{\operatorname{rank}(M)}\,
  q\!\left(M+I,\ 1+\frac{1}{z}\right).
\]
It is known that $q(M,x)$ is invariant under pivoting: whenever $M\ast A$ is defined, $q(M,x)=q(M\ast A,x)$.
Consequently, if $(M+I)\ast A$ is defined, then
\[
z^{-\operatorname{rank}(M)}P_{\langle \tau\delta\tau \rangle}(M,z)=
z^{-\operatorname{rank}(((M+I)\ast A)-I)}\,
P_{\langle \tau\delta\tau \rangle}\bigl(((M+I)\ast A)-I,\;z\bigr).
\]
\end{remark}

\section{Future research}

The partial-twualities on ribbon graphs begin with two elementary operations, $\delta$ and $\tau$, that can be applied to any subset of edges. These operations generate a group isomorphic to $S_3$, and the associated partial-twuality polynomials enumerate all partial-twualities by Euler genus.

For bouquets, we have shown that the Euler genus under any partial-twuality operation can be expressed as a linear combination of ranks and/or coranks of the adjacency matrix of the intersection graft (Theorem~\ref{lem02}). This observation led us to define, for an arbitrary square matrix $M$ over a field, five partial-$\langle\bullet\rangle$ polynomials $P_{\langle\bullet\rangle}(M,z)$ whose exponents are built from ranks and coranks of $M$, $M+I_A$, and their principal submatrices (Definition~\ref{maindef}).
This algebraic definition leads to a natural question about its underlying structure:

\begin{problem}\label{prob:matrix-action}
Do there exist two matrix operations $\delta$ and $\tau$ on pairs $(M, A)$, where $M$ is a square matrix and $A$ is a subset of its index set, that satisfy the relations $\delta^2 = \tau^2 = (\delta\tau)^3 = \mathrm{id}$, generating an $S_3$ action on such pairs, and such that for each $\bullet \in \{\delta, \tau, \delta\tau, \tau\delta, \tau\delta\tau\}$, the exponent $r_{\langle\bullet\rangle}(M, A)$ coincides with some parameter of the matrix obtained by applying $\bullet$ to $(M, A)$?
\end{problem}

Our work demonstrates that the five partial-twuality polynomials of a bouquet are precisely the five partial-$\langle\bullet\rangle$ polynomials of its associated graft. This equivalence translates questions concerning partial-twuality polynomials of a bouquet from \cite{GMT, GMT2} into questions about rank-generating functions of square matrices over a field. For $\bullet \in \{ \delta, \tau, \delta\tau, \tau\delta, \tau\delta\tau\}$, it is therefore natural to pose the following matrix-theoretic problems:

\begin{problem}
Determine conditions on a square matrix $M$ over a field such that the partial-$\langle\bullet\rangle$ polynomial of $M$ is even-interpolating, odd-interpolating, or both.
\end{problem}

\begin{problem}
Characterize those square matrices $M$ over a field for which the coefficient sequence of the partial-$\langle\bullet\rangle$ polynomial is log-concave (or unimodal).
\end{problem}

Chmutov \cite{CG3} discovered that the partial-dual Euler genus
polynomial of ribbon graphs appearing from chord diagrams is a weight system from the theory of Vassiliev knot invariants. Deng, Dong, Jin and Yan \cite{DDJY} generalized this result to the framed chord diagrams. Furthermore, Deng, Jin and Yan \cite{DJY} extended  these results to the twist polynomial of a set system by proving that the twist polynomial on set systems satisfies the four-term relation and therefore determines a finite type link invariant. Recently, Cheng \cite{Cheng} showed that the partial-dual polynomial satisfies the four-term relation of graphs, which was introduced by Lando in \cite{Lando}. Our work generalizes these polynomial invariants to a general matrix algebra framework. It is therefore natural to investigate whether the connection to the four-term relation persists at this more fundamental algebraic level.

\begin{problem}
Does there exist a matrix four-term relation, that is, a direct generalization to square matrices of the four-term relations known for chord diagrams and graphs? If such a relation exists, do the matrix partial-twuality polynomials satisfy it?
\end{problem}

\section*{Acknowledgements}
This work is supported by NSFC (Nos. 12471326, 12571379).

\end{document}